\documentclass[11pt]{siamltex1213}
\usepackage{amsmath}
\usepackage{bm}
\usepackage{amssymb,version}
\usepackage{cases}
\newtheorem{rem}{Remark}[section]
\allowdisplaybreaks
\begin{document}
    \title{Error analysis of the SAV-MAC scheme for the Navier-Stokes equations} 
\thanks{The work of X. Li is supported by the National Natural Science Foundation of China under grant number 11901489 and Postdoctoral Science Foundation of China under grant numbers BX20190187 and 2019M650152. The work of J. Shen is supported in part by NSF grants  DMS-1620262, DMS-1720442 and AFOSR  grant FA9550-16-1-0102.}

    \author{ Xiaoli Li
        \thanks{School of Mathematical Sciences and Fujian Provincial Key Laboratory on Mathematical Modeling and High Performance Scientific Computing, Xiamen University, Xiamen, Fujian, 361005, China. Email: xiaolisdu@163.com}.
        \and Jie Shen 
         \thanks{Corresponding Author. Department of Mathematics, Purdue University, West Lafayette, IN 47907, USA. Email: shen7@purdue.edu}.
}

\maketitle

\begin{abstract}
An efficient numerical scheme based on the scalar auxiliary variable (SAV) and  marker and cell scheme (MAC) is constructed for the Navier-Stokes equations. A particular feature of the scheme is that the  nonlinear term is treated explicitly while being unconditionally energy stable. A rigorous error analysis is carried out to show that both velocity and pressure approximations are second-order accurate  in time and space. Numerical experiments are presented to verify the theoretical results.  
\end{abstract}

 
\pagestyle{myheadings}
\thispagestyle{plain}
\markboth{XIAOLI LI AND JIE SHEN} {Energy Stability and Convergence of MAC Scheme}
 \section{Introduction}
 We consider in this paper the following incompressible Navier-Stokes equations:
 \begin{subequations}\label{e_model}
    \begin{align}
     \frac{\partial \textbf{u}}{\partial t}+\textbf{u}\cdot \nabla\textbf{u}
     -\nu\Delta\textbf{u}+\nabla p=\textbf{f}
     \quad &\ in\ \Omega\times J,
      \label{e_modelA}\\
      \nabla\cdot\textbf{u}=0
      \quad &\ in\ \Omega\times J,
      \label{e_modelB}\\
     \textbf{u}=\textbf{0} \quad &\ on\ \partial\Omega\times J,
      \label{e_modelC}
    \end{align}
  \end{subequations}
where $\Omega$ is an open bounded domain in $\mathbb{R}^d$ ($d=2,3$), $J=(0,T]$, $(\textbf{u},p)$ represent the unknown   velocity and pressure,  $\textbf{f}$ is an external body force,
 $\nu>0$ is the  viscosity coefficient and $\textbf{n}$ is the unit outward normal  of the domain $\Omega$. 
 
Numerical solution of the Navier-Stokes equations plays an important role in computational fluid dynamics, and  an  enormous amount of work have been devoted on the design, analysis and implementation of numerical schemes for the Navier-Stokes equations, see \cite{temam2001navier,girault2012finite,gunzburger1991analysis} and the references therein. 

One of the main difficulties in numerically solving Navier-Stokes equations is the treatment of nonlinear term. There are essentially three type of treatments: (i) fully  implicit: leads to a nonlinear system to solve at each time step; (ii) semi-implicit: needs to solve a coupled elliptic equations with variable coefficients at each time step;  and (iii) explicit: only has to solve a generalized Stokes system or even decoupled Poisson type equations at each time step, but suffers from a CFL time step constraint 
at intermediate or large Reynolds numbers. 

From a computational point of view, it would be ideal to be able to treat the nonlinear term explicitly without any stability constraint. 
In a recent work \cite{lin2019numerical}, Dong et al. constructed such a scheme by introducing an auxiliary variable. The  scheme was inspired by the recently introduced scalar auxiliary variable (SAV) approach \cite{shen2018convergence,shen2018scalar} which can lead to linear, second-order, unconditionally energy stable schemes that require  solving only decoupled elliptic equations with constant coefficients at each time step for a large class of gradient flows. The scheme constructed in \cite{lin2019numerical} for Navier-Stokes equations requires solving two generalized Stokes equations (with constant coefficient) plus a nonlinear algebraic equation for the auxiliary variable at each time step. Hence, it is very efficient compared with other existing schemes.  Ample numerical results presented in \cite{lin2019numerical} indicates that the scheme is very effective for a variety of situations. 

However,  the nonlinear algebraic equation for the auxiliary variable has multiple solutions and it is not clear whether all solutions converge to the exact solution or how to choose the right solution. This question can only be fully answered with a rigorous convergence analysis. But
 due to the explicit treatment of the nonlinear term and  the  nonlinear algebraic equation associated to the auxiliary variable, its convergence and error analysis can not be  obtained using  a standard procedure.  More precisely,
two of the main difficulties for convergence and error analysis are (i) to derive a uniform $L^\infty$ bound for the numerical solution from the modified energy stability, and (ii) to deal with the nonlinear algebraic equation for the auxiliary variable.

In this paper, we shall construct a fully discrete SAV  scheme for the Navier-Stokes equations with the marker and cell (MAC) method \cite{lebedev1964difference,Welch1965The}  for the spatial discretization.  The MAC scheme has been widely used in engineering applications due to its simplicity while satisfying  the discrete incompressibility constraint, as well as  locally conserving the mass, momentum and kinetic energy \cite{Perot2000Conservation,Perot2011Discrete}. 
The stability and  error estimates for the MAC scheme has been well studied, see for instance \cite{girault1996finite,chen2018energy,li2015superconvergence,han1998new} and  the references therein. Most of the error estimates are only   
 first order  for both the velocity and the pressure, although  Nicolaides \cite{Nicolaides1992Analysis} pointed out that numerical results suggest that the velocity is second order convergent without proof. Inspired by the  techniques in \cite{rui2012block,monk1994convergence} for Darcy-Forchheimer and Maxwell's equations, Rui and Li  established the discrete LBB condition for the MAC method and derived  second order error estimates for both the velocity and the pressure in discrete $L^2$ norms for the Stokes equations in \cite{rui2017stability,li2019superconvergence} and for the Navier-Stokes equations in  \cite{li2018superconvergence}.


The main purposes of this paper are (i) to construct  a SAV-MAC scheme for the Navier-Stokes equations, establish its energy stability,  and present an efficient algorithm for solving the resulting system which is weakly nonlinear;  (ii) to carry out a rigorous error analysis for the SAV-MAC scheme.
In particular, at each time step, our SAV-MAC scheme leads to  two discrete MAC schemes for  generalized Stokes system that can be efficiently solved by using the usual techniques developed for the MAC scheme, and  a quadratic algebraic equation for the auxiliary variable. 

The main contributions of this paper is  a rigorous error analysis with  second order  error estimates in time and space for both the velocity and pressure. This is achieved by using a bootstrap argument to establish the uniform bound for the approximate solution, followed by a sequence of delicate estimates. Our results show in particular   that at least one  solution of the quadratic algebraic equation for the auxiliary variable will converge to the exact solution. To the authors' best knowledge, this is the  first rigorous error analysis for a unconditionally energy stable scheme for the Navier-Stokes equations where  the nonlinear term is treated explicitly.

The paper is organized as follows. In Section 2,  we present the semi-discrete SAV scheme and fully discrete SAV-MAC scheme, establish the energy stability and show how to numerically solve them efficiently.  
In Section 3, we  carry  out a  rigorous error analysis to establish second order error estimates for the full discrete SAV-MAC scheme. Numerical results are presented in 
 Section 4 to  validate our theoretical results. 

We now present some notations and conventions used in the  sequel.
Throughout the paper we use $C$, with or without subscript, to denote a positive
constant, which could have different values at different appearances.

Let $L^m(\Omega)$ be the standard Banach space with norm
$$\| v\|_{L^m(\Omega)}=\left(\int_{\Omega}| v|^md\Omega\right)^{1/m}.$$
For simplicity, let
$$(f,g)=(f,g)_{L^2(\Omega)}=\int_{\Omega}fgd\Omega$$
denote the $L^2(\Omega)$ inner product,
 $\|v\|_{\infty}=\|v\|_{L^{\infty}(\Omega)}.$ And $W_p^k(\Omega)$ be the standard Sobolev space
$$W_p^k(\Omega)=\{g:~\| g\|_{W_p^k(\Omega)}<\infty\},$$
where
\begin{equation}\label{enorm1}
\| g\|_{W_p^k(\Omega)}=\left(\sum\limits_{|\alpha|\leq k}\| D^\alpha g\|_{L^p(\Omega)}^p \right)^{1/p}.
\end{equation}

  \section{The SAV-MAC scheme}
 In this section, we construct the second order MAC scheme based on the SAV approach for the Navier-Stokes equation.
 
 Define the scalar auxiliary variable $q(t)$ by 
 \begin{equation}\label{e_definition of q}
\aligned
q(t)=\sqrt{E(\textbf{u})+\delta},
\endaligned
\end{equation} 
 where $E(\textbf{u})=\int_{\Omega}\frac{1}{2}|\textbf{u}|^2$ is the total energy of the system and $\delta$ is an arbitrarily small positive constant. Then we have
\begin{equation}\label{e_equivalent form of q}
\aligned
\frac{\rm{d} q}{\rm{d} t}=\frac{1}{2q}\int_{\Omega}\frac{\partial \textbf{u}}{\partial t}\cdot \textbf{u}d\textbf{x}+\frac{1}{2\sqrt{E(\textbf{u})+\delta}}\int_{\Omega}\textbf{u}\cdot \nabla\textbf{u}\cdot \textbf{u}d\textbf{x}.
\endaligned
\end{equation}

Inspired by the work in \cite{lin2019numerical}, we transform the governing system into the following equivalent form:
  \begin{numcases}{}
 \frac{\partial \textbf{u}}{\partial t}+\frac{q(t)}{\sqrt{E(\textbf{u})+\delta}}\textbf{u}\cdot \nabla\textbf{u}
     -\nu\Delta\textbf{u}+\nabla p=\textbf{f},  \label{e_model_transform1}\\
  \frac{\rm{d} q}{\rm{d} t}=\frac{1}{2q}\int_{\Omega}\frac{\partial \textbf{u}}{\partial t}\cdot \textbf{u}d\textbf{x}+\frac{1}{2\sqrt{E(\textbf{u})+\delta}}\int_{\Omega}\textbf{u}\cdot \nabla\textbf{u}\cdot \textbf{u}d\textbf{x},     \label{e_model_transform2}\\
 \nabla\cdot\textbf{u}=0. \label{e_model_transform3}
\end{numcases}
 \subsection{The semi discrete case} For the readers' convenience, we shall first construct a second-order semi-discrete SAV scheme based on the Crank-Nicolson, although we are mainly concerned with the analysis of  a fully discrete scheme in this paper.

  Set
$$\Delta t=T/N,\ t^n=n\Delta t, \ \rm{for} \ n\leq N,$$
and define
$$[d_{t}f]^n=\frac{f^n-f^{n-1}}{\Delta t},\ \ f^{n+1/2}=\frac{f^n+f^{n+1}}{2}.$$

Then the SAV scheme based on Crank-Nicolson for \eqref{e_model_transform3} is: 
  \begin{numcases}{}
   \frac{\textbf{u}^{n+1}-\textbf{u}^{n}}{\Delta t}+\frac{q^{n+1/2}}{\sqrt{E(\tilde{\textbf{u}}^{n+1/2})+\delta}}\tilde{\textbf{u}}^{n+1/2}\cdot \nabla\tilde{\textbf{u}}^{n+1/2}
     -\nu\Delta\textbf{u}^{n+1/2}     +\nabla p^{n+1/2}=\textbf{f}^{n+1/2}, \label{e_model_semi1}\\
 \frac{q^{n+1}-q^n}{\Delta t}=\frac{1}{2q^{n+1/2}}(\frac{\textbf{u}^{n+1}-\textbf{u}^{n}}{\Delta t},\textbf{u}^{n+1/2}) \notag \\ 
   \ \ \ \ \ \ \ \ 
 +\frac{1}{2\sqrt{E(\tilde{\textbf{u}}^{n+1/2})+\delta}}(\tilde{\textbf{u}}^{n+1/2}\cdot \nabla\tilde{\textbf{u}}^{n+1/2},\textbf{u}^{n+1/2}), 
\label{e_model_semi2} \\ 
\nabla\cdot\textbf{u}^{n+1}=0, \label{e_model_semi3} 
\end{numcases}
 where 
 $\tilde{\textbf{u}}^{n+1/2}=(3\textbf{u}^{n}-\textbf{u}^{n-1})/2$ with $n\geq 1$ and we computer $\tilde{\textbf{u}}^{1/2}$ by the following simple first-order scheme:
 \begin{equation}\label{e_first order scheme}
 \aligned
 \frac{\tilde{\textbf{u}}^{1/2}-\textbf{u}^{0}}{\Delta t/2}+\textbf{u}^0\cdot \nabla\textbf{u}^0-\nu\Delta\tilde{\textbf{u}}^{1/2}+\nabla p^{1/2}=\textbf{f}^{1/2}.
  \endaligned
  \end{equation}
 which has a local truncation error of $O(\Delta t^2)$.
 
The above scheme enjoys the following stability result.
\begin{theorem}
Let $\textbf{f}\equiv 0.$ The scheme \eqref{e_model_semi1}-\eqref{e_model_semi2} is unconditionally energy stable in the sense that 
$$|q^{n+1}|^2-|q^n|^2=-\Delta t \nu \|\nabla \textbf{u}^{n+1/2}\|_{L^2}^2.$$
\end{theorem}
\begin{proof}
We recall that for $\textbf{u}\in H:=\{\textbf{u}\in L^2(\Omega): \nabla\cdot \textbf{u}=0,\; \textbf{u}\cdot \textbf{n}|_{\partial\Omega}=0\}$, we have the identity
\begin{equation}
 (\textbf{u}\cdot \textbf{v}, \textbf{v})=0\quad \forall \textbf{v}\in H^1(\Omega).
\end{equation}
 Taking the inner products of \eqref{e_model_semi1} and \eqref{e_model_semi2} with $\textbf{u}^{n+1/2}$ and $2q^{n+1/2}$, respectively, and summing up the results and using the above identity, we obtain immediately the desired result.
 \end{proof}
 
 We now describe how to solve the semi-discrete-in-time scheme
 \eqref{e_model_semi1}-\eqref{e_model_semi3} efficiently.  Inspired by the work in \cite{lin2019numerical}, we denote
 \begin{equation}\label{e_notation}
\aligned
S^{n+1}=\frac{q^{n+1/2}}{\sqrt{E(\tilde{\textbf{u}}^{n+1/2})+\delta}},\ \ \textbf{u}^{n+1}=\hat{\textbf{u}}^{n+1}+S^{n+1}\check{\textbf{u}}^{n+1},\ \ p^{n+1}=\hat{p}^{n+1}-S^{n+1}\check{p}^{n+1}.
\endaligned
\end{equation}  
Plugging in the above in \eqref{e_model_semi1} and \eqref{e_model_semi3}, we find that 
  \begin{numcases}{}
   \frac{\hat{\textbf{u}}^{n+1}}{\Delta t}-\frac{\nu}{2}\Delta\hat{\textbf{u}}^{n+1}+\nabla \hat{p}^{n+1/2}= \textbf{f}^{n+1/2}+ \frac{\textbf{u}^{n}}{\Delta t}+\frac{\nu}{2}\Delta\textbf{u}^{n}, \label{e_semi_implementation_1}\\
\nabla\cdot\hat{\textbf{u}}^{n+1}=0, \label{e_semi_implementation_2} 
\end{numcases}
  \begin{numcases}{}
   \frac{\check{\textbf{u}}^{n+1}}{\Delta t}-\frac{\nu}{2}\Delta\check{\textbf{u}}^{n+1}-\nabla \check{p}^{n+1/2}=-\tilde{\textbf{u}}^{n+1/2}\cdot \nabla\tilde{\textbf{u}}^{n+1/2}, \label{e_semi_implementation_3}\\
\nabla\cdot\check{\textbf{u}}^{n+1}=0, \label{e_semi_implementation_4} 
\end{numcases}
which are linear systems that can be solved independent of $S^{n+1}$!

It remains to  determine $S^{n+1}$. Taking the inner product of \eqref{e_model_semi1} with $\textbf{u}^{n+1/2}$, we have
\begin{equation}\label{e_semi_implementation_5}
\aligned
&( \frac{\textbf{u}^{n+1}-\textbf{u}^{n}}{\Delta t},\textbf{u}^{n+1/2})+\nu\|\nabla \textbf{u}^{n+1/2}\|^2+S^{n+1}(\tilde{\textbf{u}}^{n+1/2}\cdot \nabla\tilde{\textbf{u}}^{n+1/2},\textbf{u}^{n+1/2})\\
&\ \ \ \ \ \ \ \ \ \ 
=(\textbf{f}^{n+1/2},\textbf{u}^{n+1/2}).
\endaligned
\end{equation}
Taking the inner product of \eqref{e_model_semi2} with $2q^{n+1/2}$ leads to
\begin{equation}\label{e_semi_implementation_6}
\aligned
\frac{(q^{n+1})^2-(q^n)^2}{\Delta t}=( \frac{\textbf{u}^{n+1}-\textbf{u}^{n}}{\Delta t},\textbf{u}^{n+1/2})+S^{n+1}(\tilde{\textbf{u}}^{n+1/2}\cdot \nabla\tilde{\textbf{u}}^{n+1/2},\textbf{u}^{n+1/2}).
\endaligned
\end{equation}
Combining \eqref{e_semi_implementation_5} with \eqref{e_semi_implementation_6} results in
\begin{equation}\label{e_semi_implementation_7}
\aligned
\frac{(q^{n+1})^2-(q^n)^2}{\Delta t}+\nu\|\nabla \textbf{u}^{n+1/2}\|^2=(\textbf{f}^{n+1/2},\textbf{u}^{n+1/2}).
\endaligned
\end{equation}
Recalling \eqref{e_notation}, we find that 
\begin{equation}\label{e_semi_implementation_8}
\aligned
X_{1,n+1}(S^{n+1})^2+X_{2,n+1}S^{n+1}+X_{3,n+1}=0,
\endaligned
\end{equation}
where
\begin{flalign*}
\begin{split}
\hspace{2mm}%
&X_{1,n+1}= \frac{4}{\Delta t}(E(\tilde{\textbf{u}}^{n+1/2})+\delta)+\frac{\nu}{4}\|\nabla\check{\textbf{u}}^{n+1}\|^2,\\
\hspace{2mm}%
&X_{2,n+1}=\frac{\nu}{2}(\nabla(\hat{\textbf{u}}^{n+1}+\textbf{u}^{n}),\nabla\check{\textbf{u}}^{n+1})-
\frac{4q^n}{\Delta t}\sqrt{E(\tilde{\textbf{u}}^{n+1/2})+\delta}-\frac{1}{2}(\textbf{f}^{n+1/2},\check{\textbf{u}}^{n+1}),\\
\hspace{2mm}%
&X_{3,n+1}=\frac{\nu}{4}\|\nabla(\hat{\textbf{u}}^{n+1}+\textbf{u}^n)\|^2-\frac{1}{2}(\textbf{f}^{n+1/2},\textbf{u}^{n}+\hat{\textbf{u}}^{n+1}).
\end{split}&
\end{flalign*}  
Noting that \eqref{e_semi_implementation_8} is a quadratic equation for $S^{n+1}$ which can be solved directly by using the quadratic formula. 
Once $S^{n+1}$ is known, we can  obtain $\textbf{u}^{n+1}$ and $p^{n+1/2}$ through \eqref{e_notation}. 

\begin{rem}\label{Remark 2.1}
 The nonlinear quadratic equation  \eqref{e_semi_implementation_8} has two solutions. Since the exact solution is 1, we should choose the root which is closer to 1. In fact, to make sure that equation \eqref{e_model_semi2} makes sense, i.e., $q^{n+1/2}\ne 0$, we need to fix a constant $\kappa\in (0,1)$ and choose a root satisfying $q^{n+1/2}\ge \kappa$. 
\end{rem}
 
  \subsection{Fully discrete case} 
  

We describe below the finite difference method on the staggered grids, i.e. the MAC scheme,  for the spacial discretization of \eqref{e_model_semi1}-\eqref{e_model_semi3}. 
To fix the idea, we consider a two-dimensional rectangular domain in $\mathbb{R}^2$, i.e.,  $\Omega=(L_{lx},L_{rx})\times(L_{ly},L_{ry})$. 
We  refer to Appendix A for detailed notations about the finite difference method on the staggered grids.

Given $\{\textbf{U}^k, P^k, Q^k\}_{k=0}^{n}$, the approximations to $\{\textbf{u}^k, p^k, q^k\}_{k=0}^{n}$. We find $\{\textbf{U}^{n+1}, P^{n+1}, Q^{n+1}\}$ such that
 \begin{eqnarray}\label{e_full_discrete}
&&d_tU_1^{n+1}+\frac{Q^{n+1/2}}{B^{n+1/2}}\left(\tilde{U}_1^{n+1/2}D_x(\mathcal{P}_h\tilde{U}_1^{n+1/2})+\mathcal{P}_h \tilde{U}_2^{n+1/2}d_y(\mathcal{P}_h \tilde{U}_1^{n+1/2})\right) \notag \\
&&\hskip 2cm -\nu D_x(d_xU_1)^{n+1/2}-\nu d_y(D_yU_1)^{n+1/2}+[D_xP]^{n+1/2}=f_{1}^{n+1/2}, \label{e_full_discrete1}\\
&&d_tU_2^{n+1}+\frac{Q^{n+1/2}}{B^{n+1/2}}\left(\mathcal{P}_h \tilde{U}_1^{n+1/2}d_x(\mathcal{P}_h \tilde{U}_2^{n+1/2})+\tilde{U}_2^{n+1/2}D_y(\mathcal{P}_h\tilde{U}_2^{n+1/2})\right)  \notag \\
&&\hskip 2cm -\nu D_y(d_yU_2)^{n+1/2}-\nu d_x(D_xU_2)^{n+1/2}+[D_yP]^{n+1/2}=f_2^{n+1/2}, 
\label{e_full_discrete2}\\
&&d_tQ^{n+1}=\frac{1}{2B^{n+1/2}}(\mathcal{P}_h\tilde{\textbf{U}}^{n+1/2}\cdot\nabla_h(\mathcal{P}_h\tilde{\textbf{U}}^{n+1/2}),\textbf{U}^{n+1/2})_{l^2}\notag\\
&&\hskip 2cm +\frac{1}{2Q^{n+1/2}}(d_t\textbf{U}^{n+1},\textbf{U}^{n+1/2})_{l^2}, \label{e_full_discrete3}\\
&&d_xU_1^{n+1}+d_yU_2^{n+1}=0, \label{e_full_discrete4}
\end{eqnarray}
with  the boundary and initial conditions
\begin{eqnarray}\label{e_boundary and initial condition}
 \left\{
 \begin{array}{lll}
 \displaystyle U_{1,0,j+1/2}^{n}=U_{1,N_x,j+1/2}^{n}=0,& 0\leq j\leq N_y-1,\\
 \displaystyle U_{1,i,0}^{n}=U_{1,i,N_y}^{n}=0,& 0\leq i\leq N_x,\\
 \displaystyle U_{2,0,j}^{n}=U_{2,N_x,j}^{n}=0,  &0\leq j\leq N_y,\\
   \displaystyle U_{2,i+1/2,0}^{n}=U_{2,i+1/2,N_y}^{n}=0,& 0\leq i\leq N_x-1,\\
  \displaystyle U_{1,i,j+1/2}^{0}=u^0_{1,i,j+1/2}, &0\leq i\leq N_x,0\leq j\leq N_y,\\
  \displaystyle U_{2,i+1/2,j}^{0}=u^0_{2,i+1/2,j}, &0\leq i\leq N_x,0\leq j\leq N_y,
  \end{array}
  \right.
 \end{eqnarray}
 where $\textbf{u}^0=(u^0_1,u^0_2)$ is the initial condition. 
 In the above $B^{n+1/2}=\sqrt{E_h(\tilde{\textbf{U}}^{n+1/2})+\delta}$ with $E_h(\tilde{\textbf{U}}^{n+1/2})=\frac{1}{2}\|\tilde{\textbf{U}}^{n+1/2}\|_{l^2}^2,$ and 
 \begin{equation*}
 \aligned
(\mathcal{P}_h\tilde{\textbf{U}}^{n+1/2}&\cdot\nabla_h(\mathcal{P}_h\tilde{\textbf{U}}^{n+1/2}),\textbf{U}^{n+1/2})_{l^2}\\
=&\left(\tilde{U}_1^{n+1/2}D_x(\mathcal{P}_h\tilde{U}_1^{n+1/2})+\mathcal{P}_h \tilde{U}_2^{n+1/2}d_y(\mathcal{P}_h \tilde{U}_1^{n+1/2}), U_1^{n+1/2}\right)_{l^2,T,M}\\
&+\left(\mathcal{P}_h \tilde{U}_1^{n+1/2}d_x(\mathcal{P}_h \tilde{U}_2^{n+1/2})+\tilde{U}_2^{n+1/2}D_y(\mathcal{P}_h\tilde{U}_2^{n+1/2}),U_2^{n+1/2}\right)_{l^2,M,T},
  \endaligned
  \end{equation*}
here $\mathcal{P}_h$ is the bilinear interpolation operator. 
\

Note that the above can be efficiently solved using exactly the same procedure as in the semi-discrete case for \eqref{e_model_semi1}-\eqref{e_model_semi3}. We leave the detail to the interested readers.
In particular, $Q^{n+1}$ is determined by a quadratic algebraic equation which has two solutions. So as in the semi-discrete case Remark \ref{Remark 2.1}, we should only be concerned with the roots satisfying
\begin{equation}\label{Q_below bound}
\aligned
|Q^{n+1/2}|>\kappa
\endaligned
\end{equation} 
for a given $\kappa \in (0,1)$ and choose the root which is closer to the exact solution 1. 

  \subsection{Energy Stability} 
In this section, we will demonstrate that the second order full discrete scheme \eqref{e_full_discrete1}-\eqref{e_full_discrete4} is unconditionally energy stable. The energy stability of the semi-discrete scheme \eqref{e_model_semi1}-\eqref{e_model_semi3} can be established similarly.
 \medskip
 
 \begin{theorem}\label{thm_discrete total energy}
In the absence of the external force $\textbf{f}$, the scheme \eqref{e_full_discrete1}-\eqref{e_full_discrete4} is unconditionally stable and the following discrete energy law holds for any $\Delta t$:
\begin{equation}\label{e_discretization of energy decay}
\aligned
|Q^{n+1}|^2-|Q^{n}|^2=-\nu\Delta t \|D\textbf{U}^{n+1/2}\|^2, \ \ \forall n\geq 0.
\endaligned
\end{equation} 
\end{theorem}

\begin{proof}  
Multiplying (\ref{e_full_discrete1}) by $U_{1,i,j+1/2}^{n+1/2}hk$, making summation on $i,j$ for $1\leq i\leq N_x-1,\ 0\leq j\leq N_y-1$, and recalling Lemma \ref{lemma:U-P-Relation}, we have 
\begin{equation}\label{e_Stability1}
\aligned
&(d_tU_1^{n+1},U_1^{n+1/2})_{l^2,T,M}
+\nu\|d_x U^{n+1/2}_1\|^2_{l^2,M}
+\nu\|D_yU^{n+1/2}_1\|^2_{l^2,T_y}\\
&+\frac{Q^{n+1/2}}{B^{n+1/2}}\left(\tilde{U}_1^{n+1/2}D_x(\mathcal{P}_h\tilde{U}_1^{n+1/2})+\mathcal{P}_h \tilde{U}_2^{n+1/2}d_y(\mathcal{P}_h \tilde{U}_1^{n+1/2}), U_1^{n+1/2}\right)_{l^2,T,M}\\
&-(P^{n+1/2},d_xU^{n+1/2}_1)_{l^2,M}
=(f_1^{n+1/2},U_1^{n+1/2})_{l^2,T,M}.
\endaligned
\end{equation} 

Similarly multiplying (\ref{e_full_discrete2}) by $U_{2,i+1/2,j}^{n+1/2}hk$, and making summation on $i,j$ for $0\leq i\leq N_x-1,\ 1\leq j\leq N_y-1$, we can obtain
\begin{equation}\label{e_Stability2}
\aligned
&(d_tU_2^{n+1},U_2^{n+1/2})_{l^2,M,T}+\nu\|d_y U^{n+1/2}_2\|^2_{l^2,M}+\nu\|D_xU^{n+1/2}_2\|^2_{l^2,T_x}\\
&+\frac{Q^{n+1/2}}{B^{n+1/2}}\left(\mathcal{P}_h \tilde{U}_1^{n+1/2}d_x(\mathcal{P}_h \tilde{U}_2^{n+1/2})+\tilde{U}_2^{n+1/2}D_y(\mathcal{P}_h\tilde{U}_2^{n+1/2}),U_2^{n+1/2}\right)_{l^2,M,T}\\
&-(P^{n+1/2},d_yU^{n+1/2}_2)_{l^2,M}=(f_2^{n+1/2},U_2^{n+1/2})_{l^2,M,T}.
\endaligned
\end{equation} 

Multiplying \eqref{e_full_discrete3} by $2Q^{n+1/2}$ yields
\begin{equation}\label{e_Stability3}
\aligned
&\frac{1}{\Delta t}(|Q^{n+1}|^2-|Q^n|^2)
=\frac{Q^{n+1/2}}{B^{n+1/2}}(\mathcal{P}_h\tilde{\textbf{U}}^{n+1/2}\cdot\nabla_h\mathcal{P}_h\tilde{\textbf{U}}^{n+1/2},\textbf{U}^{n+1/2})_{l^2}\\
&\ \ \ \ \ \ \ \ \ 
+(d_t\textbf{U}^{n+1},\textbf{U}^{n+1/2})_{l^2}.
\endaligned
\end{equation}

Combining \eqref{e_Stability3} with \eqref{e_Stability1} and \eqref{e_Stability2} and 
taking notice of \eqref{e_full_discrete4} lead to
\begin{equation}\label{e_Stability4}
\aligned
|Q^{n+1}|^2&-|Q^n|^2+\nu\Delta t\|D\textbf{U}^{n+1/2}\|^2\\
=&\Delta t(f_1^{n+1/2},U_1^{n+1/2})_{l^2,T,M}+\Delta t(f_2^{n+1/2},U_2^{n+1/2})_{l^2,M,T}.
\endaligned
\end{equation}

which implies the desired result (\ref{e_discretization of energy decay}).   
\end{proof}

 \section{Error estimates} 
In this section we carry out a rigorous error analysis  for the fully discrete scheme \eqref{e_full_discrete1}-\eqref{e_full_discrete4}. More precisely, we shall prove the following main result:
In what follows, $(\textbf{u}^n,p^n,q^n)$ represents the exact solution of \eqref{e_model_transform1}-\eqref{e_model_transform3} at time $t^n$. 

\begin{theorem}\label{thm: error_estimate}
Assume that the exact solution $(\textbf{u},p)$ of \eqref{e_model_transform1}-\eqref{e_model_transform3} is sufficiently smooth such that $\textbf{u}\in W^{3}_{\infty}(J;W^{4}_{\infty}(\Omega))^2$, $p\in W^{3}_{\infty}(J;W^{3}_{\infty}(\Omega))$, denote
$(\textbf{u}^n,p^n,q^n)=(\textbf{u}(t^n),p(t^n),q(t^n))$, where  $q$ is defined in \eqref{e_definition of q}.
   Then for the fully discrete scheme \eqref{e_full_discrete1}-\eqref{e_full_discrete4} satisfying \eqref{Q_below bound} for  given $\kappa\in (0,1)$, we have the following error estimates:
   \begin{equation}\label{e_thm_error1}
\aligned
\|d_x(U_1^m-u_1^m)\|_{l^2,M}+\|d_y(U_2^m-u_2^m)\|_{l^2,M}
\leq C(\Delta t^2+h^2+k^2), \quad   \ m\leq N,
\endaligned
\end{equation}
\begin{equation}\label{e_thm_error2}
\aligned
&\|\textbf{U}^m-\textbf{u}^m\|_{l^2}+\left(\sum\limits_{l=1}^{m}\Delta t\|P^{l-1/2}-p^{l-1/2}\|^2_{l^2,M}\right)^{1/2}+|Q^m-q^m| \\
& \ \ \ \ \ \ \ \ 
\leq C(\Delta t^2+h^2+k^2),\quad   \ m\leq N,
\endaligned
\end{equation}
\begin{equation}\label{e_thm_error3}
\aligned
& \|D_y(U_1^m-u_1^m)\|_{l^2,T_y}\leq C(\Delta t^2+h^2+k^{3/2}),\quad   \ m\leq N,
\endaligned
\end{equation}
\begin{equation}\label{e_thm_error4}
\aligned
&\|D_x(U_2^m-u_2^m)\|_{l^2,T_x}\leq C(\Delta t^2+h^{3/2}+k^2),\quad   \ m\leq N.
\endaligned
\end{equation}
where the positive constant $C$ is independent of $h$, $k$ and $\Delta t$.
\end{theorem}
\begin{rem}
 The above error estimates show in particular that at least one root of the nonlinear algebraic equation \eqref{e_semi_implementation_8} will converge to the exact solution $\frac{ q(t)}{\sqrt{E(\textbf{u})+\delta}} \equiv 1$.  The numerical result presented in Fig. \ref{fig: S} clearly verifies this assertion.
\end{rem}

We shall prove the above results through a sequence of intermediate estimates below.

\subsection{An auxiliary problem}
 We consider first an auxiliary problem which will be used in the sequel.  

 Set $\textbf{g}=\textbf{f}-\textbf{u}\cdot \nabla\textbf{u}$.
 We recast \eqref{e_model} as
  \begin{subequations}\label{e_auxiliary}
    \begin{align}
   \frac{\partial \textbf{u}}{\partial t}-\nu\Delta\textbf{u}+\nabla p=\textbf{g}
     \quad &\ in\ \Omega\times J,
      \label{e_auxiliaryA}\\
      \nabla\cdot\textbf{u}=0
      \quad &\ in\ \Omega\times J,
      \label{_auxiliaryB}
    \end{align}
  \end{subequations}
and consider its approximation by the MAC scheme:
 For each $n=0,\ldots,N-1$, let $\{W^{n+1}_{1,i,j+1/2}\}$, $\{W_{2,i+1/2,j}^{n+1}\}$ and $\{H^{n+1}_{i+1/2,j+1/2}\}$ be such that
\begin{align}
&d_tW_{1,i,j+1/2}^{n+1/2}-\nu D_x(d_xW_1)^{n+1/2}_{i,j+1/2}-\nu d_y(D_yW_1)^{n+1/2}_{i,j+1/2}+[D_xH]^{n+1/2}_{i,j+1/2} \notag  \\
& \ \ \ \ \ \ \ \ \ \ \ \ \ 
=g_{1,i,j+1/2}^{n+1/2},\ \ 1\leq i\leq N_x-1,0\leq j\leq N_y-1,\label{e36}\\
&d_tW_{2,i+1/2,j}^{n+1/2}-\nu D_y(d_yW_2)^{n+1/2}_{i+1/2,j}-\nu d_x(D_xW_2)^{n+1/2}_{i+1/2,j}
+D_yH_{i+1/2,j}^{n+1/2}\notag  \\
& \ \ \ \ \ \ \ \ \ \ \ \ \ 
=g_{2,i+1/2,j}^{n+1/2},\ \ 0\leq i\leq N_x-1,1\leq j\leq N_y-1,\label{e37}\\
&d_xW^{n+1/2}_{1,i+1/2,j+1/2}+d_yW^{n+1/2}_{2,i+1/2,j+1/2}=0,\ \ 0\leq i\leq N_x-1,0\leq j\leq N_y-1,\label{e38}
\end{align}
where the boundary and initial approximations are same as \eqref{e_boundary and initial condition}.

By following closely the same arguments as in \cite{rui2017stability,li2018stability}, we can prove the following:
\begin{lemma}\label{le_auxiliary} 
Assuming that $\textbf{u}\in W^{3}_{\infty}(J;W^{4}_{\infty}(\Omega))^2$, $p\in W^{3}_{\infty}(J;W^{3}_{\infty}(\Omega))$, we have the following results:
\begin{equation}\label{e127}
\aligned
\|d_x(W^{n+1}_1-{u}^{n+1}_1)\|_{l^2,M}+\|d_y(W^{n+1}_2-{u}^{n+1}_2)\|_{l^2,M}
\leq
O(\Delta t^2+h^2+k^2),
\endaligned
\end{equation}
\begin{equation}\label{e39}
\aligned
\left(\sum\limits_{l=0}^{n}\Delta t(\|d_t(\textbf{W}^{l+1}-\textbf{u}^{l+1})\|_{l^2}^2\right)^{1/2}+\|\textbf{W}^{n+1}-\textbf{u}^{n+1}\|_{l^2}
\leq O(\Delta t^2+h^2+k^2),
\endaligned
\end{equation}

\begin{equation}\label{e132}
\aligned
& \|D_y(W^{n+1}_1-{u}^{n+1}_1)\|_{l^2,T_y}\leq O(\Delta t^2+h^2+k^{3/2}),
\endaligned
\end{equation}
\begin{equation}\label{e134}
\aligned
&\|D_x(W^{n+1}_2-{u}^{n+1}_2)\|_{l^2,T_x}\leq O(\Delta t^2+h^{3/2}+k^2),
\endaligned
\end{equation}

\begin{equation}\label{e131}
\aligned
&\left(\sum\limits_{l=1}^{N}\Delta t\|(H-p)^{l-1/2}\|^2_{l^2,M}\right)^{1/2}\leq O(\Delta t^2+h^2+k^2).
\endaligned
\end{equation}
\end{lemma}
\medskip

\subsection{Discrete LBB condition}
In order to carry out error analysis, we need the discrete LBB condition.

 Here we use the same notation and results as Rui and Li \cite[Lemma 3.3]{rui2017stability}.  
 Let $$b(\textbf{v},q)=-\int_\Omega ~q\,\nabla\cdot \textbf{v}dx,~\textbf{v}\in \textbf{V},~q\in W,$$ where
\begin{align*}
&\textbf{V}=H^1_0(\Omega)\times H^1_0(\Omega),
\quad W=\left\{q\in L^2(\Omega): \int_\Omega qdx=0\right\}.
\end{align*}

We construct the finite-dimensional subspaces of $W$ and $\textbf{V}$ by introducing three different partitions $\mathcal{T}_h,\mathcal{T}_h^1,\mathcal{T}_h^2$ of $\Omega$.
The original partition $\delta_x\times \delta_y$ is denoted by $\mathcal{T}_h$. The partition $\mathcal{T}_h^1$ is generated by connecting all the midpoints of the vertical sides of $\Omega_{i+1/2,j+1/2}$ and extending the resulting mesh to the boundary $\Gamma$. Similarly, for all $\Omega_{i+1/2,j+1/2}\in \mathcal{T}_h$ we connect all the midpoints of the horizontal sides of $\Omega_{i+1/2,j+1/2}$ and extend
the resulting mesh to the boundary $\Gamma$, then the third partition is obtained which is denoted by $\mathcal{T}_h^2$.

Corresponding to the  quadrangulation $\mathcal{T}_h$,
define $W_h$, a subspace of $W$,
 $$W_h=\left\{q_h:~q_h|_T=\text{constant},~\forall T\in \mathcal{T}_h~and \int_\Omega qdx=0\right\}.$$
  Furthermore, let $\textbf{V}_h$ be a subspace of $\textbf{V}$ such that $\textbf{V}_h$=$S_h^1\times S_h^2$, where
\begin{align*}
&S_h^l=\left\{g\in C^{(0)}(\overline{\Omega}):~g|_{T^l}\in Q_1(T^l),~,\forall T^l\in \mathcal{T}_h^l,~and~g|_\Gamma=0\right\},~l=1,2,
\end{align*}
and $Q_1$ denotes the space of all polynomials of degree $\leq 1$ with respect to each of the two variables $x$ and $y$.

We introduce the bilinear forms
$$b_h(\textbf{v}_h,q_h)=-\sum_{\Omega_{i+1/2,j+1/2}\in \mathcal{T}_h}\int_{\Omega_{i+1/2,j+1/2}} q_h \Pi_h(\nabla\cdot\textbf{v}_h)dx,~\textbf{v}_h\in \textbf{V}_h,~q_h\in W_h,$$
where
\begin{align*}
\Pi_h:~&C^{(0)}(\overline{\Omega}_{i+1/2,j+1/2})\rightarrow Q_0(\Omega_{i+1/2,j+1/2}), ~such~that\\
&(\Pi_h\varphi)_{i+1/2,j+1/2}=\varphi_{i+1/2,j+1/2},~~\forall ~\Omega_{i+1/2,j+1/2}\in \mathcal{T}_h.
\end{align*}
 Then, we have the following result \cite{rui2017stability}: 
\begin{lemma}\label{le_LBB}
There is a constant $\beta>0$, independent of $h$ and $k$ such that
\begin{equation}\label{e16}
\sup\limits_{\textbf{v}_h\in\textbf{V}_h}\frac{b_h(\textbf{v}_h,q_h)}{\|D\textbf{v}_h\|}\geq \beta\|q_h\|_{l^2,M}~~\forall q_h\in W_h.
\end{equation}
\end{lemma}

We  also define the operator $\textbf{I}_h:~\textbf{V}\rightarrow \textbf{V}_h,$ such that
\begin{equation}\label{e_H1 projection}
\aligned
(\nabla\cdot \textbf{I}_h\textbf{v},w)=(\nabla\cdot \textbf{v},w) \ \forall w\in W_h,
\endaligned
\end{equation}
with  the following approximation properties \cite{dawson1998two}:
\begin{align}
\|\textbf{v}-\textbf{I}_h\textbf{v}\|\leq &C\|\textbf{v}\|_{W^1_2(\Omega)}\hat{h}, \label{e_H1 projection_error1}\\
\|\nabla\cdot(\textbf{v}-\textbf{I}_h\textbf{v})\|\leq &C\|\nabla\cdot\textbf{v}\|_{W^1_2(\Omega)}\hat{h},\label{e_H1 projection_error2}
\end{align}
where $\hat{h}=\max\{h,k\}$.

Besides, by the definition of $\textbf{I}_h\textbf{v}$ and the midpoint rule of integration,
the $L^\infty$ norm of the projection is obtained by
\begin{equation}\label{e_H1 projection_error3}
\|\textbf{v}-\textbf{I}_h\textbf{v}\|_{\infty}\leq C\|\textbf{v}\|_{W_{\infty}^2(\Omega)}\hat{h}.
\end{equation}

Furthermore, we have the following 
 estimate \cite{duran1990superconvergence}: 
  \begin{equation}\label{e_H1 projection_error4}
\|\textbf{v}-\textbf{I}_h\textbf{v}\|_{l^2}\leq C\hat{h}^2.
\end{equation}

\subsection{A first error estimate with a $l^{\infty}_{m}(L^{\infty})$ bound assumption}

For simplicity, we set
\begin{equation}\label{errors}
\aligned
&\displaystyle e_{\textbf{u}}^n=(\textbf{U}^n-\textbf{W}^n)+(\textbf{W}^n-\textbf{u}^n):=\bm{\xi}^n+\bm{\gamma}^n,\\
& \displaystyle e_{p}^n=(P^n-H^n)+(H^n-p^n):=\eta^n+\zeta^n, \\
& \displaystyle e_{q}^n=Q^n-q^n.
\endaligned
\end{equation} 
We define $L_m$  by 
\begin{equation}\label{Lm}
L_m=\|\textbf{U}\|_{{l^{\infty}_{m}}(L^{\infty})}=\max_{n=0,\ldots,m}\|\textbf{U}^n\|_{L^{\infty}}.
 \end{equation}
First we prove the boundedness of the discrete velocity in the discrete $L^2$ norm by using the energy stability.

\begin{lemma}\label{lem: boundedness of L2 norm}
Let $\{U^k\}$ be the solution of \eqref{e_full_discrete1}-\eqref{e_full_discrete4}. We have
\begin{equation}\label{e_boundedness**}
\aligned
\|\textbf{U}^{m+1}\|_{l^2}\leq C(L_{m}),
\endaligned
\end{equation} 
where $C(L_{m})$ is independent of $h$, $k$ and $\Delta t$ but dependent of $L_{m}$.
\end{lemma}

\begin{proof}
Multiplying (\ref{e_full_discrete1}) by $d_tU_{1,i,j+1/2}^{n+1}hk$, making summation on $i,j$ for $1\leq i\leq N_x-1,\ 0\leq j\leq N_y-1$, and recalling Lemma \ref{lemma:U-P-Relation}, we have 
\begin{equation}\label{e_boundedness1}
\aligned
&\|d_tU_1^{n+1}\|_{l^2,T,M}^2+\frac{\nu}{2\Delta t}(\|d_x U^{n+1}_1\|^2_{l^2,M}-\|d_x U^{n}_1\|^2_{l^2,M}+\|D_yU^{n+1}_1\|^2_{l^2,T_y}-\|D_yU^{n}_1\|^2_{l^2,T_y})\\
&+\frac{Q^{n+1/2}}{B^{n+1/2}}\left(\tilde{U}_1^{n+1/2}D_x(\mathcal{P}_h\tilde{U}_1^{n+1/2})+\mathcal{P}_h \tilde{U}_2^{n+1/2}d_y(\mathcal{P}_h \tilde{U}_1^{n+1/2}), d_tU_1^{n+1}\right)_{l^2,T,M}\\
&-(P^{n+1/2},d_xd_tU^{n+1}_1)_{l^2,M}
=(f_1^{n+1/2},d_tU_1^{n+1/2})_{l^2,T,M}.
\endaligned
\end{equation} 

Similarly multiplying (\ref{e_full_discrete2}) by $d_tU_{2,i+1/2,j}^{n+1}hk$, and making summation on $i,j$ for $0\leq i\leq N_x-1,\ 1\leq j\leq N_y-1$, we can obtain
\begin{equation}\label{e_boundedness2}
\aligned
&\|d_tU_2^{n+1}\|_{l^2,M,T}^2+\frac{\nu}{2\Delta t}(\|d_y U^{n+1}_2\|^2_{l^2,M}-\|d_y U^{n}_2\|^2_{l^2,M}+\|D_xU^{n+1}_2\|^2_{l^2,T_x}-\|D_xU^{n}_2\|^2_{l^2,T_x})\\
&+\frac{Q^{n+1/2}}{B^{n+1/2}}\left(\mathcal{P}_h \tilde{U}_1^{n+1/2}d_x(\mathcal{P}_h \tilde{U}_2^{n+1/2})+\tilde{U}_2^{n+1/2}D_y(\mathcal{P}_h\tilde{U}_2^{n+1/2}),d_tU_2^{n+1}\right)_{l^2,M,T}\\
&-(P^{n+1/2},d_yd_tU^{n+1}_2)_{l^2,M}=(f_2^{n+1/2},d_tU_2^{n+1})_{l^2,M,T}.
\endaligned
\end{equation} 
Combining \eqref{e_boundedness1} with \eqref{e_boundedness2} results in 
\begin{equation}\label{e_boundedness3}
\aligned
\|d_t\textbf{U}^{n+1}\|_{l^2}^2&+\frac{\nu}{2\Delta t}(\|D\textbf{U}^{n+1}\|^2-\|D\textbf{U}^{n}\|^2)\\
=&(\textbf{f}^{n+1/2},d_t\textbf{U}^{n+1})_{l^2}-\frac{Q^{n+1/2}}{B^{n+1/2}}(\mathcal{P}_h\tilde{\textbf{U}}^{n+1/2}\cdot\nabla_h(\mathcal{P}_h\tilde{\textbf{U}}^{n+1/2}),d_t\textbf{U}^{n+1})_{l^2}.
\endaligned
\end{equation} 
Recalling \eqref{e_Stability4} and using Cauchy-Schwarz inequality and Poincar\'e inequality, we  obtain
\begin{equation}\label{e_boundedness4}
\aligned
&|Q^{n+1}|^2-|Q^{0}|^2+\nu\sum\limits_{k=0}^n\Delta t\|D\textbf{U}^{k+1/2}\|^2
=\sum\limits_{k=0}^n\Delta t(\textbf{f}^{k+1/2},\textbf{U}^{k+1/2})\\
&\ \ \ \ \ \ \ \ 
\leq \frac{\nu}{2}\sum\limits_{k=0}^n\Delta t\|D\textbf{U}^{k+1/2}\|^2+C\sum\limits_{k=0}^n\Delta t\|\textbf{f}^{k+1/2}\|_{l^2}^2,
\endaligned
\end{equation} 
which implies 
\begin{equation}\label{e_boundedness of Q}
\aligned
|Q^{n+1}|\leq C.
\endaligned
\end{equation}
Using the above and the assumption, the last term on the right hand side of \eqref{e_boundedness3} can be estimated by
\begin{equation}\label{e_boundedness5}
\aligned
|-\frac{Q^{n+1/2}}{B^{n+1/2}}&(\mathcal{P}_h\tilde{\textbf{U}}^{n+1/2}\cdot\nabla_h(\mathcal{P}_h\tilde{\textbf{U}}^{n+1/2}),d_t\textbf{U}^{n+1})_{l^2}|\\
\leq &C(L_{n})(\|D\textbf{U}^{n}\|^2+\|D\textbf{U}^{n-1}\|^2)+\frac{1}{4}\|d_t\textbf{U}^{n+1}\|_{l^2}^2.
\endaligned
\end{equation} 
Combining \eqref{e_boundedness3} with \eqref{e_boundedness5} and using Cauchy-Schwarz ineqality, we have
\begin{equation}\label{e_boundedness6}
\aligned
\|d_t\textbf{U}^{n+1}\|_{l^2}^2&+\frac{\nu}{2\Delta t}(\|D\textbf{U}^{n+1}\|^2-\|D\textbf{U}^{n}\|^2)\\
\leq&C(L_{n})(\|D\textbf{U}^{n}\|^2+\|D\textbf{U}^{n-1}\|^2)+\frac{1}{2}\|d_t\textbf{U}^{n+1}\|_{l^2}^2+\frac{1}{2}\|\textbf{f}^{n+1/2}\|_{l^2}^2.
\endaligned
\end{equation} 
Multiplying \eqref{e_boundedness6} by $2\Delta t$, summing over $n$ from 0 to $m$ and applying Gronwall inequality give that
\begin{equation}\label{e_boundedness7}
\aligned
&\|D\textbf{U}^{m+1}\|^2\leq C(L_{m})\sum\limits_{n=0}^{m}\Delta t\|\textbf{f}^{n+1/2}\|_{l^2}^2.
\endaligned
\end{equation} 
Thus we can get the desired result \eqref{e_boundedness**} by applying the discrete Poincar\'e inequality.
\end{proof}
\medskip

\begin{lemma}\label{lem: error_estimate_u}
Assuming $\textbf{u}\in W^{3}_{\infty}(J;W^{4}_{\infty}(\Omega))^2$, $p\in W^{3}_{\infty}(J;W^{3}_{\infty}(\Omega))$, we have
\begin{equation}\label{e_error_estimate_u}
\aligned
\frac{1}{2}\|\bm{\xi}^{m+1}\|_{l^2}^2&+\frac{\nu}{2}\sum_{n=0}^m\Delta t\|D \bm{\xi}^{n+1/2}\|^2+|e_{q}^{m+1}|^2\\
\leq &C(L_{m})\sum_{n=0}^m\Delta t\|\bm\xi^{n+1}\|^2_{l^2}
+\frac{1}{2}\sum_{n=0}^m\Delta t\|d_t\bm\xi^{n+1}\|_{l^2}^2\\
&+\frac{1}{\kappa}C(L_{m})\sum_{n=0}^m\Delta t|e_q^{n+1}|^2+\frac{1}{\kappa}C(L_{m})(\Delta t^4+h^4+k^4),
\endaligned
\end{equation}
where $\bm{\xi}^k$ and $e_{q}^{k}$ are defined in \eqref{errors},
$\kappa$ is the constant in \eqref{Q_below bound}, and   the positive constant $C(L_{m})$ is independent of $h$, $k$ and $\Delta t$ but dependent of $L_{m}$.
\end{lemma}

\begin{proof}
Subtracting \eqref{e36} from \eqref{e_full_discrete1}, we obtain
\begin{equation}\label{e_lem_error5}
\aligned
&d_t\xi^{n+1}_{1,i,j+1/2}
-\nu D_x(d_x\xi_1)^{n+1/2}_{i,j+1/2}-\nu d_y(D_y\xi_1)^{n+1/2}_{i,j+1/2}\\
&\ \ \ \ \ \ \ \ 
+[D_x\eta]^{n+1/2}_{i,j+1/2}
=T_{1,i,j+1/2}^{n+1/2},
\endaligned
\end{equation}
where 
\begin{equation}\label{e_T1}
\aligned
T_{1}^{n+1/2}=&-\frac{Q^{n+1/2}}{B^{n+1/2}}\left(\tilde{U}_1^{n+1/2}D_x(\mathcal{P}_h\tilde{U}_1^{n+1/2})+\mathcal{P}_h \tilde{U}_2^{n+1/2}d_y(\mathcal{P}_h \tilde{U}_1^{n+1/2})\right)\\
&+\frac{q(t^{n+1/2})}{\sqrt{E(\textbf{u}^{n+1/2})+\delta}}\left(u_1\frac{\partial u_1}{\partial x}+u_2\frac{\partial u_1}{\partial y}\right)^{n+1/2}.
\endaligned
\end{equation}
Subtracting \eqref{e37} from \eqref{e_full_discrete2}, we obtain
\begin{equation}\label{e_lem_error6}
\aligned
&d_t\xi^{n+1}_{2,i+1/2,j}
-\nu D_y(d_y\xi_2)^{n+1/2}_{i+1/2,j}-\nu d_x(D_x\xi_2)^{n+1/2}_{i+1/2,j}\\
&\ \ \ \ \ \ \ \ 
+[D_y\eta]^{n+1/2}_{i+1/2,j}=T_{2,i+1/2,j}^{n+1/2},
\endaligned
\end{equation}
where 
\begin{equation}\label{e_T2}
\aligned
T_{2}^{n+1/2}=&-\frac{Q^{n+1/2}}{B^{n+1/2}}\left(\mathcal{P}_h \tilde{U}_1^{n+1/2}d_x(\mathcal{P}_h \tilde{U}_2^{n+1/2})+\tilde{U}_2^{n+1/2}D_y(\mathcal{P}_h\tilde{U}_2^{n+1/2})\right)\\
&+\frac{q(t^{n+1/2})}{\sqrt{E(\textbf{u}^{n+1/2})+\delta}}\left(u_1\frac{\partial u_2}{\partial x}+u_2\frac{\partial u_2}{\partial y}\right)^{n+1/2}.
\endaligned
\end{equation}
Subtracting \eqref{e_model_transform2} from \eqref{e_full_discrete3}, we obtain
\begin{equation}\label{e_lem_error7}
\aligned
d_te_q^{n+1}=\frac{1}{2B^{n+1/2}}(\mathcal{P}_h\tilde{\textbf{U}}^{n+1/2}\cdot\nabla_h(\mathcal{P}_h\tilde{\textbf{U}}^{n+1/2}),\bm{\xi}^{n+1/2})_{l^2}+\sum\limits_{k=1}^{3}S_k^{n+1/2},
\endaligned
\end{equation}
where
\begin{flalign*}
\begin{split}
\hspace{10mm}%
&S_{1}^{n+1/2}= \frac{\rm{d} q^{n+1/2}}{\rm{d} t}-d_tq^{n+1},\\
\hspace{10mm}%
&S_{2}^{n+1/2}=\frac{1}{2B^{n+1/2}}(\mathcal{P}_h\tilde{\textbf{U}}^{n+1/2}\cdot\nabla_h(\mathcal{P}_h\tilde{\textbf{U}}^{n+1/2}),\textbf{W}^{n+1/2})_{l^2}\\
&\ \ \ \ \ \ 
-\frac{1}{2\sqrt{E(\textbf{u}^{n+1/2})+\delta}}\int_{\Omega}\textbf{u}^{n+1/2}\cdot \nabla\textbf{u}^{n+1/2}\cdot \textbf{u}^{n+1/2}d\textbf{x},\\
&S_{3}^{n+1/2}=\frac{1}{2Q^{n+1/2}}(d_t\textbf{U}^{n+1},\textbf{U}^{n+1/2})_{l^2}-
\frac{1}{2q^{n+1/2}}\int_{\Omega}\frac{\partial \textbf{u}^{n+1/2}}{\partial t}\cdot \textbf{u}^{n+1/2}d\textbf{x}.
\end{split}&
\end{flalign*}
Multiplying \eqref{e_lem_error5} by $\xi_{1,i,j+1/2}^{n+1/2}hk$, making summation on $i,j$ for $1\leq i\leq N_x-1,\ 0\leq j\leq N_y-1$ and applying Lemma \ref{lemma:U-P-Relation}, we have 
\begin{equation}\label{e_lem_error8}
\aligned
&(d_t\xi^{n+1}_{1},\xi_{1}^{n+1/2})_{l^2,T,M}+\nu\|d_x \xi^{n+1/2}_1\|^2_{l^2,M}+\nu\|D_y\xi^{n+1/2}_1\|^2_{l^2,T_y}\\
&-(\eta^{n+1/2},d_x\xi^{n+1/2}_1)_{l^2,M}
=(T_1^{n+1/2},\xi_1^{n+1/2})_{l^2,T,M}.
\endaligned
\end{equation}
Multiplying (\ref{e_lem_error6}) by $\xi_{2,i+1/2,j}^{n+1/2}hk$, making summation on $i,j$ for $0\leq i\leq N_x-1,\ 1\leq j\leq N_y-1$ and applying Lemma \ref{lemma:U-P-Relation} lead to
\begin{equation}\label{e_lem_error9}
\aligned
&(d_t\xi^{n+1}_{2},\xi_{2}^{n+1/2})_{l^2,M,T}+\nu\|d_y \xi^{n+1/2}_2\|^2_{l^2,M}
+\nu\|D_x\xi^{n+1/2}_2\|^2_{l^2,T_x}\\
&-(\eta^{n+1/2},d_y\xi^{n+1/2}_2)_{l^2,M}
=(T_2^{n+1/2},\xi_2^{n+1/2})_{l^2,M,T}.
\endaligned
\end{equation}
Multiplying equation (\ref{e_lem_error7}) by $(e_{q}^{n+1}+e_{q}^{n})$ leads to 
\begin{equation}\label{e_lem_error10}
\aligned
&\frac{(e_{q}^{n+1})^2-(e_{q}^{n})^2}{\Delta t}=\frac{e_{q}^{n+1/2}}{B^{n+1/2}}(\mathcal{P}_h\tilde{\textbf{U}}^{n+1/2}\cdot\nabla_h(\mathcal{P}_h\tilde{\textbf{U}}^{n+1/2}),\bm{\xi}^{n+1/2})_{l^2}\\
&\ \ \ \ \ \ \ \ 
+2\sum\limits_{k=3}^{5}S_k^{n+1/2}e_{q}^{n+1/2}.
\endaligned
\end{equation}
Combining \eqref{e_lem_error8} with \eqref{e_lem_error9}, we have 
\begin{equation}\label{e_lem_error11}
\aligned
&(d_t\bm{\xi}^{n+1},\bm{\xi}^{n+1/2})_{l^2}+\nu  \|D \bm{\xi}^{n+1/2}\|^2-(\eta^{n+1/2},d_x\xi^{n+1/2}_1+d_y\xi^{n+1/2}_2)_{l^2,M}\\
& \ \ \ \ \ \ \ 
=(\textbf{T}^{n+1/2},\bm{\xi}^{n+1/2})_{l^2},
\endaligned
\end{equation}
where $\textbf{T}=(T_1,T_2)$.
Subtracting \eqref{e38} from \eqref{e_full_discrete4}, we obtain
\begin{equation}\label{e_lem_error12}
\aligned
d_x\xi_1^{n+1}+d_y\xi_2^{n+1}=0,\ \ 0\leq i\leq N_x-1,0\leq j\leq N_y-1.
\endaligned
\end{equation}
Thus we have 
\begin{equation}\label{e_lem_error13}
\aligned
(\eta^{n+1/2},d_x\xi^{n+1/2}_1+d_y\xi^{n+1/2}_2)_{l^2,M}=0.
\endaligned
\end{equation}
The term on the right hand side of \eqref{e_lem_error11} can be recast as 
\begin{equation}\label{e_lem_error14}
\aligned
(\textbf{T}^{n+1/2},\bm{\xi}^{n+1/2})_{l^2}=&-\frac{e_q^{n+1/2}}{B^{n+1/2}}(\mathcal{P}_h\tilde{\textbf{U}}^{n+1/2}\cdot\nabla_h(\mathcal{P}_h\tilde{\textbf{U}}^{n+1/2}),\bm{\xi}^{n+1/2})_{l^2}\\
&-\frac{q^{n+1/2}}{B^{n+1/2}}(\mathcal{P}_h\tilde{\textbf{U}}^{n+1/2}\cdot\nabla_h(\mathcal{P}_h\tilde{\textbf{U}}^{n+1/2}),\bm{\xi}^{n+1/2})_{l^2}\\
&+\frac{q^{n+1/2}}{\sqrt{E(\textbf{u}^{n+1/2})+\delta}}(\textbf{u}^{n+1/2}\cdot\nabla\textbf{u}^{n+1/2},\bm{\xi}^{n+1/2})_{l^2}.
\endaligned
\end{equation}
The last two terms on the right hand side of \eqref{e_lem_error14} can be transformed into the following:
\begin{equation}\label{e_lem_error15}
\aligned
&\frac{q^{n+1/2}}{\sqrt{E(\textbf{u}^{n+1/2})+\delta}}(\textbf{u}^{n+1/2}\cdot\nabla\textbf{u}^{n+1/2},\bm{\xi}^{n+1/2})_{l^2}\\
&-\frac{q^{n+1/2}}{B^{n+1/2}}(\mathcal{P}_h\tilde{\textbf{U}}^{n+1/2}\cdot\nabla_h(\mathcal{P}_h\tilde{\textbf{U}}^{n+1/2}),\bm{\xi}^{n+1/2})_{l^2}\\
\leq &(\textbf{u}^{n+1/2}\cdot\nabla\textbf{u}^{n+1/2},\bm{\xi}^{n+1/2})_{l^2}(\frac{q^{n+1/2}}{\sqrt{E(\textbf{u}^{n+1/2})+\delta}}-\frac{q^{n+1/2}}{B^{n+1/2}})\\
&-\frac{q^{n+1/2}}{B^{n+1/2}}(\mathcal{P}_h\tilde{\textbf{U}}^{n+1/2}\cdot\nabla_h(\mathcal{P}_h\tilde{\textbf{u}}^{n+1/2})-\textbf{u}^{n+1/2}\cdot \nabla \textbf{u}^{n+1/2},\bm{\xi}^{n+1/2})_{l^2}\\
&-\frac{q^{n+1/2}}{B^{n+1/2}}(\mathcal{P}_h\tilde{\textbf{U}}^{n+1/2}\cdot\nabla_h(\mathcal{P}_h\tilde{\bm\gamma}^{n+1/2}),\bm{\xi}^{n+1/2})_{l^2}\\
&-\frac{q^{n+1/2}}{B^{n+1/2}}(\mathcal{P}_h\tilde{\textbf{U}}^{n+1/2}\cdot\nabla_h(\mathcal{P}_h\tilde{\bm\xi}^{n+1/2}),\bm{\xi}^{n+1/2})_{l^2}.
\endaligned
\end{equation}
Recalling the midpoint approximation property of the rectangle quadrature formula and using the Cauchy-Schwarz inequality, the first term on the right hand side of \eqref{e_lem_error15} can be estimated as
\begin{equation}\label{e_lem_error16}
\aligned
(\textbf{u}^{n+1/2}&\cdot\nabla\textbf{u}^{n+1/2},\bm{\xi}^{n+1/2})_{l^2}(\frac{q^{n+1/2}}{\sqrt{E(\textbf{u}^{n+1/2})+\delta}}-\frac{q^{n+1/2}}{B^{n+1/2}})\\
\leq &C(L_{n}) q^{n+1/2}\|\bm\xi^{n+1/2}\|_{l^2}|E_h(\tilde{\textbf{U}}^{n+1/2})-E(\textbf{u}^{n+1/2})|\\
\leq &C(L_{n})(\|\bm\xi^{n+1}\|^2_{l^2}+\|\bm\xi^{n}\|^2_{l^2}+\|\bm\xi^{n-1}\|^2_{l^2})\\
&+C(L_{n})(\|\bm\gamma^{n}\|^2_{l^2}+\|\bm\gamma^{n-1}\|^2_{l^2})+C(L_{n})(\Delta t^4+h^4+k^4).
\endaligned
\end{equation}
Using the Cauchy-Schwarz inequality, the second term on the right hand side of \eqref{e_lem_error15} can be estimated as
\begin{equation}\label{e_lem_error17}
\aligned
-\frac{q^{n+1/2}}{B^{n+1/2}}&(\mathcal{P}_h\tilde{\textbf{U}}^{n+1/2}\cdot\nabla_h(\mathcal{P}_h\tilde{\textbf{u}}^{n+1/2})-\textbf{u}^{n+1/2}\cdot \nabla \textbf{u}^{n+1/2},\bm{\xi}^{n+1/2})_{l^2}\\
= &-\frac{q^{n+1/2}}{B^{n+1/2}}\left((\mathcal{P}_h\tilde{\textbf{U}}^{n+1/2}-\textbf{u}^{n+1/2})
\cdot\nabla_h(\mathcal{P}_h\tilde{\textbf{u}}^{n+1/2}),\bm{\xi}^{n+1/2}\right)_{l^2}\\
&-\frac{q^{n+1/2}}{B^{n+1/2}}(
\textbf{u}^{n+1/2}\cdot \nabla (\mathcal{P}_h\tilde{\textbf{u}}^{n+1/2}-\textbf{u}^{n+1/2}),\bm{\xi}^{n+1/2})_{l^2}\\
\leq &C(L_{n})(\|\bm\xi^{n+1}\|^2_{l^2}+\|\bm\xi^{n}\|^2_{l^2}+\|\bm\xi^{n-1}\|^2_{l^2})\\
&+C(L_{n})(\|\bm\gamma^{n}\|^2_{l^2}+\|\bm\gamma^{n-1}\|^2_{l^2})+C(L_{n})(\Delta t^4+h^4+k^4).
\endaligned
\end{equation}
Recalling Lemma \ref{lemma:U-P-Relation}, the third term on the right hand side of \eqref{e_lem_error15} can be controlled by
\begin{equation}\label{e_lem_error18}
\aligned
-\frac{q^{n+1/2}}{B^{n+1/2}}&(\mathcal{P}_h\tilde{\textbf{U}}^{n+1/2}\cdot\nabla_h(\mathcal{P}_h\tilde{\bm\gamma}^{n+1/2}),\bm{\xi}^{n+1/2})_{l^2}\\
\leq &C(L_{n})|(\nabla_h(\mathcal{P}_h\tilde{\bm\gamma}^{n+1/2}),\bm{\xi}^{n+1/2})_{l^2}|\\
\leq &\frac{\nu}{4}\|D\bm\xi^{n+1/2}\|^2+C(L_n)(\Delta t^4+h^4+k^4).
\endaligned
\end{equation}
The last term on the right hand side of \eqref{e_lem_error15} can be bounded by
\begin{equation}\label{e_lem_error19}
\aligned
-\frac{q^{n+1/2}}{B^{n+1/2}}&(\mathcal{P}_h\tilde{\textbf{U}}^{n+1/2}\cdot\nabla_h(\mathcal{P}_h\tilde{\bm\xi}^{n+1/2}),\bm{\xi}^{n+1/2})_{l^2}\\
\leq &\frac{\nu}{4}\|D\bm\xi^{n+1/2}\|^2+C(L_{n})\|\bm\xi^{n+1/2}\|^2_{l^2}+C(L_{n})(h^4+k^4).
\endaligned
\end{equation}
Combining \eqref{e_lem_error11} with \eqref{e_lem_error12}-\eqref{e_lem_error19} results in
\begin{equation}\label{e_lem_error20}
\aligned
\frac{\|\bm{\xi}^{n+1}\|_{l^2}^2-\|\bm{\xi}^{n}\|_{l^2}^2}{2\Delta t}&+\nu  \|D \bm{\xi}^{n+1/2}\|^2+\frac{e_q^{n+1/2}}{B^{n+1/2}}(\mathcal{P}_h\tilde{\textbf{U}}^{n+1/2}\cdot\nabla_h(\mathcal{P}_h\tilde{\textbf{U}}^{n+1/2}),\bm{\xi}^{n+1/2})_{l^2}\\
\leq &C(L_{n})(\|\bm\xi^{n+1}\|^2_{l^2}+\|\bm\xi^{n}\|^2_{l^2}+\|\bm\xi^{n-1}\|^2_{l^2})+C(L_{n})(\|\bm\gamma^{n}\|^2_{l^2}+\|\bm\gamma^{n-1}\|^2_{l^2})\\
&+\frac{\nu}{2}\|D\bm\xi^{n+1/2}\|^2+C(L_{n})(\Delta t^4+h^4+k^4).
\endaligned
\end{equation}
Next we estimate the last term on the right hand side of \eqref{e_lem_error10} by
\begin{equation}\label{e_lem_error21}
\aligned
&2S_1^{n+1/2}e_{q}^{n+1/2}\leq C(|e_q^{n+1}|^2+|e_q^{n}|^2)+C\|q\|_{W^{3}_{\infty}(J)}^2\Delta t^4.
\endaligned
\end{equation}

\begin{equation}\label{e_lem_error22}
\aligned
2S_2^{n+1/2}e_{q}^{n+1/2}= &\frac{e_{q}^{n+1/2}}{B^{n+1/2}}(\mathcal{P}_h\tilde{\textbf{U}}^{n+1/2}\cdot\nabla_h(\mathcal{P}_h(\tilde{\bm\xi}^{n+1/2}+\tilde{\bm\gamma}^{n+1/2}),\textbf{W}^{n+1/2})_{l^2}\\
&+\frac{e_{q}^{n+1/2}}{B^{n+1/2}}(\mathcal{P}_h\tilde{\textbf{U}}^{n+1/2}\cdot\nabla_h(\mathcal{P}_h\tilde{\textbf{u}}^{n+1/2}),\textbf{W}^{n+1/2})_{l^2}\\
&-\frac{e_{q}^{n+1/2}}{\sqrt{E(\textbf{u}^{n+1/2})+\delta}}\int_{\Omega}\textbf{u}^{n+1/2}\cdot \nabla\textbf{u}^{n+1/2}\cdot \textbf{u}^{n+1/2}d\textbf{x}.
\endaligned
\end{equation}
The analysis of the first term on the right hand side of \eqref{e_lem_error22} can be carried out with the help of Lemmas \ref{lemma:U-P-Relation} and \ref{le_auxiliary}:
\begin{equation}\label{e_lem_error24}
\aligned
\frac{e_{q}^{n+1/2}}{B^{n+1/2}}&(\mathcal{P}_h\tilde{\textbf{U}}^{n+1/2}\cdot\nabla_h(\mathcal{P}_h(\tilde{\bm\xi}^{n+1/2}+\tilde{\bm\gamma}^{n+1/2}),\textbf{W}^{n+1/2})_{l^2}\\
\leq & C(L_{n})|e_{q}^{n+1/2}||(\nabla_h(\mathcal{P}_h(\tilde{\bm\xi}^{n+1/2}+\tilde{\bm\gamma}^{n+1/2}),\textbf{W}^{n+1/2})_{l^2}|\\
\leq & C(L_{n})|e_{q}^{n+1/2}|\|\tilde{\bm\xi}^{n+1/2}+\tilde{\bm\gamma}^{n+1/2}\|_{l^2}\|D\textbf{W}^{n+1/2}\|\\
\leq &C(L_{n})|e_{q}^{n+1/2}|^2+C(L_{n})(\|\bm\xi^{n}\|_{l^2}^2+\|\bm\xi^{n-1}\|_{l^2}^2)\\
&+C(L_{n})(\Delta t^4+h^4+k^4).
\endaligned
\end{equation}
The last two terms on the right hand side of \eqref{e_lem_error22} can be handled in a similar way as \eqref{e_lem_error16}:
\begin{equation}\label{e_lem_error25}
\aligned
\frac{e_{q}^{n+1/2}}{B^{n+1/2}}&(\mathcal{P}_h\tilde{\textbf{U}}^{n+1/2}\cdot\nabla_h(\mathcal{P}_h\tilde{\textbf{u}}^{n+1/2}),\textbf{W}^{n+1/2})_{l^2}\\
&-\frac{e_{q}^{n+1/2}}{\sqrt{E(\textbf{u}^{n+1/2})+\delta}}\int_{\Omega}\textbf{u}^{n+1/2}\cdot \nabla\textbf{u}^{n+1/2}\cdot \textbf{u}^{n+1/2}d\textbf{x}\\
\leq &C(L_{n})(\|\bm\xi^{n}\|^2_{l^2}+\|\bm\xi^{n-1}\|^2_{l^2})+
C(L_{n})|e_{q}^{n+1/2}|^2\\
&+C(L_{n})(\|\bm\gamma^{n}\|^2_{l^2}+\|\bm\gamma^{n-1}\|^2_{l^2})\\
&+C(L_{n})(\Delta t^4+h^4+k^4),
\endaligned
\end{equation}
where we use the fact that $\|\textbf{W}^{n+1/2}\|_{\infty}\leq C$ with the aid of Lemma 
 \ref{le_auxiliary} and the inverse assumption.

Recalling Lemma \ref{lem: boundedness of L2 norm} and using Cauchy-Schwarz inequality, we have
\begin{equation}\label{e_lem_error23}
\aligned
2S_3^{n+1/2}e_{q}^{n+1/2}= &\frac{e_{q}^{n+1/2}}{Q^{n+1/2}}(d_t\textbf{U}^{n+1},\textbf{U}^{n+1/2})_{l^2}-
\frac{e_{q}^{n+1/2}}{q^{n+1/2}}\int_{\Omega}\frac{\partial \textbf{u}^{n+1/2}}{\partial t}\cdot \textbf{u}^{n+1/2}d\textbf{x}\\
=&\frac{e_{q}^{n+1/2}}{Q^{n+1/2}}(d_t(\bm\xi^{n+1}+\bm\gamma^{n+1}),\textbf{U}^{n+1/2})_{l^2}
+\frac{e_{q}^{n+1/2}}{Q^{n+1/2}}(d_t\textbf{u}^{n+1},\bm\xi^{n+1/2}+\bm\gamma^{n+1/2})_{l^2}\\
&+\frac{e_{q}^{n+1/2}}{Q^{n+1/2}}(d_t\textbf{u}^{n+1},\textbf{u}^{n+1/2})_{l^2}-\frac{e_{q}^{n+1/2}}{q^{n+1/2}}\int_{\Omega}\frac{\partial \textbf{u}^{n+1/2}}{\partial t}\cdot \textbf{u}^{n+1/2}d\textbf{x}\\
\leq & \frac{1}{\kappa}C(L_n)|e_{q}^{n+1/2}|^2+\frac{1}{2}\|d_t\bm\xi^{n+1}\|_{l^2}^2+\frac{1}{\kappa}C(L_n)\|d_t\bm\gamma^{n+1}\|_{l^2}^2\\
&+C\|\bm\xi^{n+1}\|_{l^2}^2+C\|\bm\gamma^{n+1}\|_{l^2}^2+\frac{1}{\kappa}C(L_n)(\Delta t^4+h^4+k^4).
\endaligned
\end{equation}
Combining \eqref{e_lem_error10} with \eqref{e_lem_error21}-\eqref{e_lem_error23} leads to
\begin{equation}\label{e_lem_error26}
\aligned
&\frac{(e_{q}^{n+1})^2-(e_{q}^{n})^2}{\Delta t}
\leq \frac{e_{q}^{n+1/2}}{B^{n+1/2}}(\mathcal{P}_h\tilde{\textbf{U}}^{n+1/2}\cdot\nabla_h(\mathcal{P}_h\tilde{\textbf{U}}^{n+1/2}),\bm{\xi}^{n+1/2})_{l^2}\\
&\ \ \ \ \ \ \ 
+C(L_{n})(|e_q^{n+1}|^2+|e_q^{n}|^2)
+\frac{1}{2}\|d_t\bm\xi^{n+1}\|_{l^2}^2\\
&\ \ \ \ \ \ \ 
+C(L_{n})(\|\bm\xi^{n}\|_{l^2}^2+\|\bm\xi^{n-1}\|_{l^2}^2)\\
&\ \ \ \ \ \ \ 
+C(L_n)(\Delta t^4+h^4+k^4).
\endaligned
\end{equation}
Then by combining \eqref{e_lem_error20} with \eqref{e_lem_error26}, we can obtain
\begin{equation}\label{e_lem_error27}
\aligned
\frac{\|\bm{\xi}^{n+1}\|_{l^2}^2-\|\bm{\xi}^{n}\|_{l^2}^2}{2\Delta t}&+\frac{\nu}{2}\|D \bm{\xi}^{n+1/2}\|^2+\frac{(e_{q}^{n+1})^2-(e_{q}^{n})^2}{\Delta t}\\
\leq &C(L_{n})(\|\bm\xi^{n+1}\|^2_{l^2}+\|\bm\xi^{n}\|^2_{l^2}+\|\bm\xi^{n-1}\|^2_{l^2})
+\frac{1}{2}\|d_t\bm\xi^{n+1}\|_{l^2}^2\\
&+C(L_{n})(|e_q^{n+1}|^2+|e_q^{n}|^2)+C(L_{n})(\Delta t^4+h^4+k^4).
\endaligned
\end{equation}
Then we can obtain the desired result \eqref{e_error_estimate_u} by multiplying equation (\ref{e_lem_error27}) by $\Delta t$ and summing over $n$ from $0$ to $m$.
\end{proof}
\medskip

\begin{lemma}\label{lem: error_estimate_dtu}
Assuming $\textbf{u}\in W^{3}_{\infty}(J;W^{4}_{\infty}(\Omega))^2$, $p\in W^{3}_{\infty}(J;W^{3}_{\infty}(\Omega))$, then we have
\begin{equation}\label{e_error_estimate_dtu}
\aligned
\sum_{n=0}^m\Delta t\|d_t\bm{\xi}^{n+1}\|_{l^2}^2+\frac{\nu}{2}\|D \bm{\xi}^{m+1}\|^2
\leq &C(L_{m})\sum_{n=0}^m\Delta t\|\bm\xi^{n}\|^2_{l^2}+C(L_{m})\sum_{n=0}^m\Delta t
|e_q^{n+1/2}|^2\\
&+C(L_{m})\sum_{n=0}^m\Delta t\|D\bm\xi^{n}\|^2+C(L_{m})(\Delta t^4+h^4+k^4).
\endaligned
\end{equation}
where  $\bm{\xi}^k$ and $e_{q}^{k}$ are defined in \eqref{errors}, and the positive constant $C(L_{m})$ is independent of $h$, $k$ and $\Delta t$ but dependent of $L_{m}$.
\end{lemma}

\begin{proof}
Multiplying (\ref{e_lem_error5}) by $d_t\xi_{1,i,j+1/2}^{n+1}hk$, making summation on $i,j$ for $1\leq i\leq N_x-1,\ 0\leq j\leq N_y-1$ and applying Lemma \ref{lemma:U-P-Relation}, we have 
\begin{equation}\label{e_lemdtu_1}
\aligned
&\|d_t\xi^{n+1}_{1}\|_{l^2,T,M}^2+\frac{\nu}{2}\frac{\|d_x \xi^{n+1}_1\|^2_{l^2,M}-\|d_x \xi^{n}_1\|^2_{l^2,M}}{\Delta t}+\frac{\nu}{2}\frac{\|D_y \xi^{n+1}_1\|^2_{l^2,T_y}-\|D_y \xi^{n}_1\|^2_{l^2,T_y}}{\Delta t}\\
&-(\eta^{n+1/2},d_xd_t\xi^{n+1}_1)_{l^2,M}
=(T_1^{n+1/2},d_t\xi_1^{n+1})_{l^2,T,M}.
\endaligned
\end{equation}
Multiplying (\ref{e_lem_error6}) by $d_t\xi_{2,i+1/2,j}^{n+1}hk$, making summation on $i,j$ for $0\leq i\leq N_x-1,\ 1\leq j\leq N_y-1$ and applying Lemma \ref{lemma:U-P-Relation} lead to
\begin{equation}\label{e_lemdtu_2}
\aligned
&\|d_t\xi^{n+1}_{2}\|^2_{l^2,M,T}+\frac{\nu}{2}\frac{\|d_y \xi^{n+1}_2\|^2_{l^2,M}-\|d_y \xi^{n}_2\|^2_{l^2,M}}{\Delta t}+\frac{\nu}{2}\frac{\|D_x \xi^{n+1}_2\|^2_{l^2,T_x}-\|D_x \xi^{n}_2\|^2_{l^2,T_x}}{\Delta t}\\
&-(\eta^{n+1/2},d_yd_t\xi^{n+1}_2)_{l^2,M}
=(T_2^{n+1/2},d_t\xi_2^{n+1})_{l^2,M,T}.
\endaligned
\end{equation}
Combining \eqref{e_lemdtu_1} with \eqref{e_lemdtu_2}, we have 
\begin{equation}\label{e_lemdtu_3}
\aligned
&\|d_t\bm{\xi}^{n+1}\|_{l^2}^2+\frac{\nu}{2}\frac{\|D \bm{\xi}^{n+1}\|^2-\|D \bm{\xi}^{n}\|^2}{\Delta t}=(\textbf{T}^{n+1/2},d_t\bm{\xi}^{n+1})_{l^2}.
\endaligned
\end{equation}
The right hand side of \eqref{e_lemdtu_3} can be estimated as 
\begin{equation}\label{e_lemdtu_4}
\aligned
(\textbf{T}^{n+1/2},d_t\bm{\xi}^{n+1})_{l^2}
=&(\frac{q^{n+1/2}}{\sqrt{E(\textbf{u}^{n+1/2})+\delta}}-\frac{Q^{n+1/2}}{B^{n+1/2}})
(\textbf{u}^{n+1/2}\cdot\nabla_h\textbf{u}^{n+1/2},d_t\bm{\xi}^{n+1})_{l^2}\\
&-\frac{Q^{n+1/2}}{B^{n+1/2}}(\mathcal{P}_h\tilde{\textbf{U}}^{n+1/2}\cdot\nabla_h\mathcal{P}_h\tilde{\textbf{u}}^{n+1/2}-\textbf{u}^{n+1/2}\cdot\nabla_h\textbf{u}^{n+1/2},d_t\bm{\xi}^{n+1})_{l^2}\\
&-\frac{Q^{n+1/2}}{B^{n+1/2}}(\mathcal{P}_h\tilde{\textbf{U}}^{n+1/2}\cdot\nabla_h\mathcal{P}_h\tilde{\bm\xi}^{n+1/2},d_t\bm{\xi}^{n+1})_{l^2}\\
&-\frac{Q^{n+1/2}}{B^{n+1/2}}(\mathcal{P}_h\tilde{\textbf{U}}^{n+1/2}\cdot\nabla_h\mathcal{P}_h\tilde{\bm\gamma}^{n+1/2},d_t\bm{\xi}^{n+1})_{l^2}.
\endaligned
\end{equation}
The first term on the right hand side of \eqref{e_lemdtu_4} can be handled in a similar way as \eqref{e_lem_error16}:
\begin{equation}\label{e_lemdtu_5}
\aligned
(\frac{q^{n+1/2}}{\sqrt{E(\textbf{u}^{n+1/2})+\delta}}&-\frac{Q^{n+1/2}}{B^{n+1/2}})
(\textbf{u}^{n+1/2}\cdot\nabla_h\textbf{u}^{n+1/2},d_t\bm{\xi}^{n+1})_{l^2}\\
=&(\frac{q^{n+1/2}}{\sqrt{E(\textbf{u}^{n+1/2})+\delta}}-\frac{q^{n+1/2}}{B^{n+1/2}})
(\textbf{u}^{n+1/2}\cdot\nabla_h\textbf{u}^{n+1/2},d_t\bm{\xi}^{n+1})_{l^2}\\
&-\frac{e_q^{n+1/2}}{B^{n+1/2}}(\textbf{u}^{n+1/2}\cdot\nabla_h\textbf{u}^{n+1/2},d_t\bm{\xi}^{n+1})_{l^2}\\
\leq &\frac{1}{6}\|d_t\bm\xi^{n+1}\|^2_{l^2}+C(L_{n})(\|\bm\xi^{n}\|^2_{l^2}+\|\bm\xi^{n-1}\|^2_{l^2})\\
&+C(L_{n})(\|\bm\gamma^{n}\|^2_{l^2}+\|\bm\gamma^{n-1}\|^2_{l^2})+C(L_{n})|e_q^{n+1/2}|^2\\
&+C(L_{n})(\Delta t^4+h^4+k^4).
\endaligned
\end{equation}
Using \eqref{e_boundedness of Q} and the definition of $\mathcal{P}_h$, we can estimate 
the second term on the right hand side of \eqref{e_lemdtu_4} as 
\begin{equation}\label{e_lemdtu_6}
\aligned
-\frac{Q^{n+1/2}}{B^{n+1/2}}&(\mathcal{P}_h\tilde{\textbf{U}}^{n+1/2}\cdot\nabla_h\mathcal{P}_h\tilde{\textbf{u}}^{n+1/2}-\textbf{u}^{n+1/2}\cdot\nabla_h\textbf{u}^{n+1/2},d_t\bm{\xi}^{n+1})_{l^2}\\
\leq &C(L_{n})(\|\bm\xi^{n}\|^2_{l^2}+\|\bm\xi^{n-1}\|^2_{l^2})+\frac{1}{6}\|d_t\bm\xi^{n+1}\|^2_{l^2}\\
&+C(L_{n})(\|\bm\gamma^{n}\|^2_{l^2}+\|\bm\gamma^{n-1}\|^2_{l^2})+C(L_{n})(\Delta t^4+h^4+k^4).
\endaligned
\end{equation}
Applying Cauchy-Schwarz inequality, the third term on the right hand side of \eqref{e_lemdtu_4} can be controlled by
\begin{equation}\label{e_lemdtu_7}
\aligned
-\frac{Q^{n+1/2}}{B^{n+1/2}}&(\mathcal{P}_h\tilde{\textbf{U}}^{n+1/2}\cdot\nabla_h\mathcal{P}_h\tilde{\bm\xi}^{n+1/2},d_t\bm{\xi}^{n+1})_{l^2}\\
\leq& C(L_{n})(\|D\bm\xi^{n}\|^2+\|D\bm\xi^{n-1}\|^2)+\frac{1}{6}\|d_t\bm\xi^{n+1}\|^2_{l^2}\\
&+C(L_n)(h^4+k^4).
\endaligned
\end{equation}
Combining \eqref{e_lemdtu_3} with \eqref{e_lemdtu_4}-\eqref{e_lemdtu_7} yields
\begin{equation}\label{e_lemdtu_8}
\aligned
\|d_t\bm{\xi}^{n+1}\|_{l^2}^2&+\frac{\nu}{2}\frac{\|D \bm{\xi}^{n+1}\|^2-\|D \bm{\xi}^{n}\|^2}{\Delta t}\\
\leq &-\frac{Q^{n+1/2}}{B^{n+1/2}}(\mathcal{P}_h\tilde{\textbf{U}}^{n+1/2}\cdot\nabla_h\mathcal{P}_h\tilde{\bm\gamma}^{n+1/2},d_t\bm{\xi}^{n+1})_{l^2}\\
&+\frac{1}{2}\|d_t\bm\xi^{n+1}\|^2_{l^2}+C(L_{n})(\|\bm\xi^{n}\|^2_{l^2}+\|\bm\xi^{n-1}\|^2_{l^2})\\
&+C(L_{n})(\|\bm\gamma^{n}\|^2_{l^2}+\|\bm\gamma^{n-1}\|^2_{l^2})+C(L_{n})|e_q^{n+1/2}|^2\\
&+C(L_{n})(\|D\bm\xi^{n}\|^2+\|D\bm\xi^{n-1}\|^2)\\
&+C(L_{n})(\Delta t^4+h^4+k^4).
\endaligned
\end{equation}
Multiplying equation (\ref{e_lemdtu_8}) by $2\Delta t$ and summing over $n$ from $0$ to $m$, we have
\begin{equation}\label{e_lemdtu_9}
\aligned
\sum_{n=0}^m&\Delta t\|d_t\bm{\xi}^{n+1}\|_{l^2}^2+\nu\|D \bm{\xi}^{m+1}\|^2\\
\leq &-2\sum_{n=0}^m\Delta t\frac{Q^{n+1/2}}{B^{n+1/2}}(\mathcal{P}_h\tilde{\textbf{U}}^{n+1/2}\cdot\nabla_h\mathcal{P}_h\tilde{\bm\gamma}^{n+1/2},d_t\bm{\xi}^{n+1})_{l^2}\\
&+C(L_{m})\sum_{n=0}^m\Delta t\|\bm\xi^{n}\|^2_{l^2}+C(L_{m})\sum_{n=0}^m\Delta t
|e_q^{n+1/2}|^2\\
&+C(L_{m})\sum_{n=0}^m\Delta t\|D\bm\xi^{n}\|^2+C(L_{m})(\Delta t^4+h^4+k^4).
\endaligned
\end{equation}
From the discrete-integration-by-parts, the first term on the right hand side of \eqref{e_lemdtu_9} can be transformed into 
\begin{equation}\label{e_lemdtu_10}
\aligned
-2\sum_{n=0}^m&\Delta t\frac{Q^{n+1/2}}{B^{n+1/2}}(\mathcal{P}_h\tilde{\textbf{U}}^{n+1/2}\cdot\nabla_h\mathcal{P}_h\tilde{\bm\gamma}^{n+1/2},d_t\bm{\xi}^{n+1})_{l^2}\\
\leq &C(L_{m})|\sum_{n=0}^m\Delta t(\nabla_h\mathcal{P}_h\tilde{\bm\gamma}^{n+1/2},d_t\bm{\xi}^{n+1})_{l^2}|\\
\leq & C(L_{m})|(\nabla_h\mathcal{P}_h\tilde{\bm\gamma}^{m+1/2},\bm{\xi}^{m+1})_{l^2}-\sum_{n=1}^m\Delta t(\nabla_hd_t\mathcal{P}_h\tilde{\bm\gamma}^{n+1/2},\bm{\xi}^{n})_{l^2}|\\
\leq &C(L_{m})\sum_{n=1}^m\Delta t\|d_t\tilde{\bm\gamma}^{n+1/2}\|_{l^2}^2+C(L_{m})\sum_{n=1}^m\Delta t\|D\bm\xi^{n}\|^2\\
&+\frac{\nu}{2}\|D\bm\xi^{m+1}\|_{l^2}^2+C(L_{m})(\Delta t^4+h^4+k^4).
\endaligned
\end{equation}
Substituting \eqref{e_lemdtu_10} into \eqref{e_lemdtu_9} leads to
\begin{equation}\label{e_lemdtu_11}
\aligned
\sum_{n=0}^m&\Delta t\|d_t\bm{\xi}^{n+1}\|_{l^2}^2+\frac{\nu}{2}\|D \bm{\xi}^{m+1}\|^2\\
\leq &C(L_{m})\sum_{n=0}^m\Delta t\|\bm\xi^{n}\|^2_{l^2}+C(L_{m})\sum_{n=0}^m\Delta t
|e_q^{n+1/2}|^2\\
&+C(L_{m})\sum_{n=0}^m\Delta t\|D\bm\xi^{n}\|^2+C(L_{m})(\Delta t^4+h^4+k^4).
\endaligned
\end{equation}
\end{proof}
\medskip

\begin{lemma}\label{lem: error_estimate_p}
Assuming $\textbf{u}\in W^{3}_{\infty}(J;W^{4}_{\infty}(\Omega))^2$, $p\in W^{3}_{\infty}(J;W^{3}_{\infty}(\Omega))$, we have
\begin{equation}\label{e_error_estimate_p}
\aligned
\sum_{n=0}^{m}\Delta t\|\eta^{n+1/2}\|_{l^2,M}^2
\leq& C(L_{m})\sum_{n=0}^{m}\Delta t\|d_t\bm\xi^{n+1}\|_{l^2}^2\\
&+C(L_m)\sum_{n=0}^{m}\Delta t\|D\bm\xi^{n+1/2}\|^2+
C(L_{m})\sum_{n=0}^{m}\Delta t\|\bm\xi^{n}\|^2_{l^2}\\
&+C(L_{m})\sum_{n=0}^m\Delta t
|e_q^{n+1/2}|^2+C(L_{m})(\Delta t^4+h^4+k^4).
\endaligned
\end{equation}
where $\eta^k,\,\bm{\xi}^k$ and $e_{q}^{k}$ are defined in \eqref{errors}, and  the positive constant $C(L_{m})$ is independent of $h$, $k$ and $\Delta t$ but dependent of $L_{m}$.
\end{lemma}

\begin{proof}
For a discrete function $\{v^{n+1/2}_{1,i,j+1/2}\}$ such that $v^{n+1/2}_{1,i,j+1/2}|_{\partial \Omega}=0$,  multiplying (\ref{e_lem_error5}) by times $v^{n+1/2}_{1,i,j+1/2}hk$
and make summation for $i,j$ with $i=1,\cdots,N_x-1,~j=0,\cdots,N_y-1$, and recalling Lemma \ref{lemma:U-P-Relation} lead to 
\begin{equation}\label{e_errorp_1}
\aligned
&(d_t\xi^{n+1}_{1},v_{1}^{n+1/2})_{l^2,T,M}+\nu(d_x \xi^{n+1/2}_1,d_x v^{n+1/2}_1)_{l^2,M}+\nu(D_y\xi^{n+1/2}_1,D_yv^{n+1/2}_1)_{l^2,T_y}\\
&-(\eta^{n+1/2},d_xv^{n+1/2}_1)_{l^2,M}
=(T_1^{n+1/2},v_1^{n+1/2})_{l^2,T,M}.
\endaligned
\end{equation}
Similarly in the $y$ direction, we can obtain
\begin{equation}\label{e_errorp_2}
\aligned
&(d_t\xi^{n+1}_{2},v_{2}^{n+1/2})_{l^2,M,T}+\nu(d_y \xi^{n+1/2}_2,d_y v^{n+1/2}_2)_{l^2,M}
+\nu(D_x\xi^{n+1/2}_2,D_xv^{n+1/2}_2)_{l^2,T_x}\\
&-(\eta^{n+1/2},d_yv^{n+1/2}_2)_{l^2,M}
=(T_2^{n+1/2},v_2^{n+1/2})_{l^2,M,T}.
\endaligned
\end{equation}
Combining \eqref{e_errorp_1} with \eqref{e_errorp_2} results in
\begin{equation}\label{e_errorp_3}
\aligned
&(d_t\bm\xi^{n+1},\textbf{v}^{n+1/2})_{l^2}+\nu(D\bm\xi^{n+1/2},D\textbf{v}^{n+1/2})\\
&-(\eta^{n+1/2},d_xv^{n+1/2}_1+d_yv^{n+1/2}_2)_{l^2,M}\\
&=(\textbf{T}^{n+1/2},\textbf{v}^{n+1/2})_{l^2}.
\endaligned
\end{equation}
Using Lemma \ref{le_LBB} and \eqref{e_lem_error14}, and the discrete Poincar$\acute{e}$ inequality, we can obtain
\begin{equation}\label{e_errorp_4}
\aligned
&\beta\|\eta^{n+1/2}\|_{l^2,M}
\leq\sup\limits_{\textbf{v}\in \textbf{V}_h}\frac{(\eta^{n+1/2},d_xv^{n+1/2}_1+d_yv^{n+1/2}_2)_{l^2,M}}{\|D\textbf{v}^{n+1/2}\|}\\
& \ \ \ \ \ \
\leq C\|d_t\bm\xi^{n+1}\|_{l^2}+C\|D\bm\xi^{n+1/2}\|+
C(L_{n})(\|\bm\xi^{n}\|_{l^2}+\|\bm\xi^{n-1}\|_{l^2})\\
& \ \ \ \ \ \ +C(L_{n})|e_q^{n+1/2}|
+C(L_{n})(\|\bm\gamma^{n}\|_{l^2}+\|\bm\gamma^{n-1}\|_{l^2})\\
& \ \ \ \ \ \
+C(L_{n})(\Delta t^2+h^2+k^2).
\endaligned
\end{equation}
Then we can obtain the desired result \eqref{e_error_estimate_p}.
\end{proof}
\medskip

Combing the above results together, we obtain the following results under the $l^{\infty}_{m}(L^{\infty})$ bound assumption:
\begin{lemma}\label{lem: error_estimate_up}
Assuming $\textbf{u}\in W^{3}_{\infty}(J;W^{4}_{\infty}(\Omega))^2$, $p\in W^{3}_{\infty}(J;W^{3}_{\infty}(\Omega))$, we have
\begin{equation}\label{e_error_estimate_up}
\aligned
\|\bm{\xi}^{m+1}\|_{l^2}^2&+\sum_{n=0}^m\Delta t\|d_t\bm{\xi}^{n+1}\|_{l^2}^2+\|D \bm{\xi}^{m+1}\|^2+\sum_{n=0}^{m}\Delta t\|\eta^{n+1/2}\|_{l^2,M}^2
+|e_{q}^{m+1}|^2\\
\leq &C(L_{m})(\Delta t^4+h^4+k^4).
\endaligned
\end{equation}
where $\eta^k,\,\bm{\xi}^k$ and $e_{q}^{k}$ are defined in \eqref{errors}, and   the positive constant $C(L_{m})$ is independent of $h$, $k$ and $\Delta t$ but dependent of $L_m$.
\end{lemma}

\begin{proof}
Combining \eqref{e_error_estimate_u} with \eqref{e_error_estimate_dtu}, we have
\begin{equation}\label{e_error_estimate_up1}
\aligned
\frac{1}{2}\|\bm{\xi}^{m+1}\|_{l^2}^2&+\frac{1}{2}\sum_{n=0}^m\Delta t\|d_t\bm{\xi}^{n+1}\|_{l^2}^2+\frac{\nu}{2}\|D \bm{\xi}^{m+1}\|^2+|e_{q}^{m+1}|^2\\
\leq&C(L_{m})\sum_{n=0}^m\Delta t\|\bm\xi^{n+1}\|^2_{l^2}+C(L_{m})\sum_{n=0}^m\Delta t|e_q^{n+1}|^2\\
&+C(L_{m})\sum_{n=0}^m\Delta t\|D\bm\xi^{n}\|^2+C(L_{m})(\Delta t^4+h^4+k^4).
\endaligned
\end{equation}
Then applying the discrete Gronwall's inequality, we arrive at the desired result:
\begin{equation}\label{e_error_estimate_up2}
\|\bm{\xi}^{m+1}\|_{l^2}^2+\sum_{n=0}^m\Delta t\|d_t\bm{\xi}^{n+1}\|_{l^2}^2+\|D \bm{\xi}^{m+1}\|^2
+|e_{q}^{m+1}|^2\leq C(L_{m})(\Delta t^4+h^4+k^4).
\end{equation}
Recalling \eqref{e_error_estimate_p}, we have
\begin{equation}\label{e_error_estimate_up3}
\aligned
&\sum_{n=0}^{m}\Delta t\|\eta^{n+1/2}\|_{l^2,M}^2
\leq C(L_{m})\sum_{n=0}^{m}\Delta t\|d_t\bm\xi^{n+1}\|_{l^2}^2+
C(L_{m})\sum_{n=0}^{m}\Delta t\|\bm\xi^{n}\|^2_{l^2}\\
&\ \ \ 
+C(L_{m})\sum_{n=0}^m\Delta t|e_q^{n+1}|^2+C(L_m)\sum_{n=0}^{m}\Delta t\|D\bm\xi^{n+1/2}\|^2+C(L_{m})(\Delta t^4+h^4+k^4)\\
&\ \ \ 
\leq C(L_{m})(\Delta t^4+h^4+k^4).
\endaligned
\end{equation}
\end{proof}

\subsection{Verification of the $l^{\infty}_{m}(L^{\infty})$ bound assumption}
{}\vskip 10pt
 
\begin{lemma}\label{lem: boundedness of maxnorm}
Under the assumptions of Theorem 3.1, there exists a positive constant $C_1$ independent of $h$, $k$ and $\Delta t$ such that 
\begin{equation}\label{e_boundedness of maxnorm U}
\aligned
\|\textbf{U}^{m}\|_{\infty}\leq C_1, \ for \ all \ 0\le m\leq N=T/{\Delta t}.
\endaligned
\end{equation} 
\end{lemma}

\begin{proof} 
We proceed in the following two steps using a bootstrap argument.

\textit{\textbf{Step 1}} (Definition of $C_1$): Using the scheme (\ref{e_full_discrete1})-(\ref{e_full_discrete4}) for $n=0$, Lemma \ref{lem: error_estimate_up}, properties of the operator $\textbf{I}_h$ and the inverse assumption, we can get the approximation $\textbf{U}^1$ and the following property:
\begin{equation*}
\aligned
\|\textbf{U}^1\|_{\infty}=& \|\textbf{U}^1-\textbf{I}_h\textbf{u}^1\|_{\infty}+\|\textbf{I}_h\textbf{u}^1-\textbf{u}^1\|_{\infty}+\|\textbf{u}^1\|_{\infty}\\
\leq&C\hat{h}^{-1}\|\textbf{U}^1-\textbf{I}_h\textbf{u}^1\|_{l^2}+\|\textbf{I}_h\textbf{u}^1-\textbf{u}^1\|_{\infty}+\|\textbf{u}^1\|_{\infty}\\
\leq& C\hat{h}^{-1}(\|\bm\xi^1+\bm\gamma^1\|_{l^2}+\|\textbf{I}_h\textbf{u}^1-\textbf{u}^1\|_{l^2})+\|\textbf{I}_h\textbf{u}^1-\textbf{u}^1\|_{\infty}+\|\textbf{u}^1\|_{\infty}\\
\leq& C\hat{h}^{-1}(\Delta t^2+\hat{h}^2)+\|\textbf{u}^1\|_{\infty}\leq C.
\endaligned
\end{equation*}
where $\hat{h}$ and $\Delta t$ are selected such that $\hat{h}^{-1}\Delta t^2$ is sufficiently small.

 Thus define the positive constant $C_1$ independent of $\hat{h}$ and $\Delta t$ such that
\begin{align*}
C_1&\geq \max\{\|\textbf{U}^1\|_{\infty}, 2\|\textbf{u}^{n}\|_{\infty}\}.
\end{align*}

\textit{\textbf{Step 2}} (Induction): We can easily obtain that hypothesis (\ref{e_boundedness of maxnorm U}) holds true for $m=1$ by the definition of $C_1$.  Supposing that $\|\textbf{U}^{m}\|_{\infty}\leq C_1$ holds true for an integer $m=1,\cdots,N-1$ and using Lemma \ref{lem: error_estimate_up}, we obtain
$$\|\bm\xi^{m+1}\|_{l^2}\leq C(L_m)(\Delta t^2+\hat{h}^2).$$
Next we prove that $\|\textbf{U}^{m+1}\|_{\infty}\leq C_1$ holds true.
Since
\begin{equation}\label{e_hypothesis_proof1}
\aligned
\|\textbf{U}^{m+1}\|_{\infty}=& \|\textbf{U}^{m+1}-\textbf{I}_h\textbf{u}^{m+1}\|_{\infty}+\|\textbf{I}_h\textbf{u}^{m+1}-\textbf{u}^{m+1}\|_{\infty}+\|\textbf{u}^{m+1}\|_{\infty}\\
\leq& C\hat{h}^{-1}(\|\bm\xi^{m+1}+\bm\gamma^{m+1}\|_{l^2}+\|\textbf{I}_h\textbf{u}^{m+1}-\textbf{u}^{m+1}\|_{l^2})\\
&+\|\textbf{I}_h\textbf{u}^{m+1}-\textbf{u}^{m+1}\|_{\infty}+\|\textbf{u}^{m+1}\|_{\infty}\\
\leq& C_2\hat{h}^{-1}(\Delta t^2+\hat{h}^2)+\|\textbf{u}^{m+1}\|_{\infty}.
\endaligned
\end{equation}
Let $\Delta t\leq C_3\hat{h}$ and a positive constant $\hat{h}_*$ be small enough to satisfy
$$C_2(1+C_3^2)\hat{h}_*\leq\frac{C_1}{2}.$$
Then for $\hat{h}\in (0,\hat{h}_*],$ equation (\ref{e_hypothesis_proof1}) can be controlled by
\begin{equation}\label{e_hypothesis_proof2}
\aligned
\|\textbf{U}^{m+1}\|_{\infty}\leq& C_2\hat{h}^{-1}(\Delta t^2+\hat{h}^2)+\|\textbf{u}^{m+1}\|_{\infty}\\
\leq&C_2(1+C_3^2)\hat{h}_*+\frac{C_1}{2}\leq C_1.
\endaligned
\end{equation}
Then the proof of induction hypothesis (\ref{e_boundedness of maxnorm U}) ends.
\end{proof}
\medskip

{\it Proof of Theorem \ref{thm: error_estimate}}.
Finally, 
 using Lemmas \ref{le_auxiliary}, \ref{lem: error_estimate_up} and \ref{lem: boundedness of maxnorm}, we  arrive at the conclusions of  Theorem \ref{thm: error_estimate}.
 \framebox(4,6){}

 \section{Numerical experiments} \label{Numerical experiments}
 In this section, 
 we provide some numerical results to  verify the accuracy of the proposed numerical scheme.

 We take $\Omega=(0,1)\times (0,1)$, $T=1$, $\nu=1$ and $\delta=0.1$, and set $\Delta t=h=k$. We denote
\begin{flalign*}
\renewcommand{\arraystretch}{1.5}
  \left\{
   \begin{array}{l}
\|e_X\|_{\infty,2}=\max\limits_{0\leq n\leq m}\left\{\|e_{X}^{n}\|\right\},\\
\|e_p\|_{2,2}=\left(\sum\limits_{n=0}^{m}\Delta t\left\|P^{n+1/2}-p^{n+1/2}\right\|_{l^2,M}^2\right)^{1/2},\\
\|e_q\|_{\infty}=\max\limits_{0\leq n\leq m}|Q^{n}-q^{n}|,
\end{array}\right.
\end{flalign*}
where $X=\textbf{u}, d_xu_1, D_yu_1$.

{\bf Example 1}. The right hand side of the equations are computed according to the analytic solution given as below:\\
\begin{equation*}
\aligned
\begin{cases}
p(x,y,t)={\rm exp}(t)(x^3-1/4),\\
u_1(x,y,t)=-{\rm  exp}(t)x^2(x-1)^2y(y-1)(2y-1)/256,\\
u_2(x,y,t)={\rm exp}(t)x(x-1)(2x-1)y^2(y-1)^2/256.
\end{cases}
\endaligned
\end{equation*}
The numerical results for Example 1 are presented in Tables \ref{table1_example1}-\ref{table1_example2}. We observe that the results are  consistent with the error estimates in Theorem \ref{thm: error_estimate}.

{\bf Example 2}. The right hand side of the equations are computed according to the analytic solution given as below:\\
\begin{equation*}
\aligned
\begin{cases}
p(x,y,t)={\rm exp}(t)(\sin(\pi y)-2/\pi),\\
u_1(x,y,t)={\rm exp}(t)\sin^2(\pi x)\sin(2\pi y),\\
u_2(x,y,t)=-{\rm exp}(t)\sin(2\pi x)\sin^2(\pi y).
\end{cases}
\endaligned
\end{equation*}

The numerical results for Example 2 are presented in Tables \ref{table2_example1}-\ref{table2_example2}. We observe uniform second-order convergence for all quantities, including  $D_yu_1$ for which  Theorem \ref{thm: error_estimate} predicts only 3/2-order convergence. This is due to  the fact that, for this particular exact solution, we have $\frac{\partial^2 u^x}{\partial y^2}=0$ for $y=0$ and $y=1$ and $\frac{\partial^2 u^y}{\partial x^2}=0$ for $x=0$ and $x=1$, which lead to a super-convergence for $D_yu_1$ (see related results in \cite{rui2017stability,li2018superconvergence}). 

Note that we only presented the results for $u_1$ in both examples since the results for $u_2$ are similar to $u_1$.

\begin{table}[htbp]
\renewcommand{\arraystretch}{1.1}
\small
\centering
\caption{Convergence rates of the velocity for Example 1.}\label{table1_example1}
\begin{tabular}{p{1.5cm}p{1.3cm}p{1.2cm}p{1.7cm}p{1.2cm}p{2.0cm}p{1.2cm}}\hline
$N_x\times N_y$    &$\|e_{\textbf{u}}\|_{\infty,2}$    &Rate &$\|e_{d_xu_1}\|_{\infty,2}$   &Rate  
&$\|e_{D_yu_1}\|_{\infty,2}$    &Rate   \\ \hline
$2^4\times 2^4$     &1.05E-6               & ---    &2.78E-6         &---  &8.71E-6         &---\\
$2^5\times 2^5$     &2.59E-7            & 2.02    &6.82E-7         &2.03  &3.21E-6        &1.44\\
$2^6\times 2^6$      &6.41E-8           & 2.01    &1.65E-7         &2.04 &1.16E-6         &1.47\\
$2^7\times 2^7$    &1.59E-8            &2.01     &4.01E-8         &2.05 &4.16E-7       &1.48\\
\hline
\end{tabular}
\end{table}

\begin{table}[htbp]
\renewcommand{\arraystretch}{1.1}
\small
\centering
\caption{Convergence rates of the pressure and auxiliary variable for Example 1.}\label{table2_example1}
\begin{tabular}{p{1.5cm}p{1.3cm}p{1.2cm}p{1.8cm}p{1.2cm}}\hline
$N_x\times N_y$    &$\|e_{p}\|_{2,2}$    &Rate &$\|e_{q}\|_{\infty}$   &Rate   \\ \hline
$2^4\times 2^4$     &1.01E-3               & ---    &5.10E-11         &---  \\
$2^5\times 2^5$     &2.52E-4               & 2.00    &1.36E-11         &1.90  \\
$2^6\times 2^6$      &6.30E-5              & 2.00    &3.44E-12         &1.99 \\
$2^7\times 2^7$      &1.57E-5              &2.00     &8.57E-13         &2.00 \\
\hline
\end{tabular}
\end{table}

\begin{table}[htbp]
\renewcommand{\arraystretch}{1.1}
\small
\centering
\caption{Convergence rates of the velocity for Example 2.}\label{table1_example2}
\begin{tabular}{p{1.5cm}p{1.3cm}p{1.2cm}p{1.7cm}p{1.2cm}p{1.8cm}p{1.2cm}}\hline
$N_x\times N_y$    &$\|e_{\textbf{u}}\|_{\infty,2}$    &Rate &$\|e_{d_xu_1}\|_{\infty,2}$   &Rate  
&$\|e_{D_yu_1}\|_{\infty,2}$    &Rate   \\ \hline
$2^4\times 2^4$     &2.15E-2               & ---    &4.94E-2         &---  &9.53E-2         &---\\
$2^5\times 2^5$     &5.21E-3           & 2.05    &1.28E-2         &1.94  &2.31E-2         &2.04\\
$2^6\times 2^6$      &1.28E-3           &2.02    &3.29E-3         &1.96 &5.70E-3         &2.02\\
$2^7\times 2^7$     &3.18E-4            &2.01     &8.20E-4         &2.01 &1.41E-3        &2.01\\
\hline
\end{tabular}
\end{table}

\begin{table}[htbp]
\renewcommand{\arraystretch}{1.1}
\small
\centering
\caption{Convergence rates of the pressure and auxiliary variable for Example 2.}\label{table2_example2}
\begin{tabular}{p{1.5cm}p{1.3cm}p{1.2cm}p{1.5cm}p{1.2cm}}\hline
$N_x\times N_y$    &$\|e_{p}\|_{2,2}$    &Rate &$\|e_{q}\|_{\infty}$   &Rate   \\ \hline
$2^4\times 2^4$     &6.38E-2               & ---    &1.35E-2         &---  \\
$2^5\times 2^5$     &1.42E-2               & 2.17    &3.49E-3         &1.95  \\
$2^6\times 2^6$      &3.27E-3              & 2.12    &8.72E-4         &2.00 \\
$2^7\times 2^7$      &7.97E-4              &2.04     &2.17E-4         &2.01 \\
\hline
\end{tabular}
\end{table}

\medskip
{\bf Example 3}. We take the initial condition to be $u_{1}^0(x,y)=\sin^2(\pi x)\sin(2\pi y)$, $u_{2}^0(x,y)=\sin(2\pi x)\sin^2(\pi y)$ and ${\bf f}=0$. 

We present in  Fig. \ref{fig: S} the time evolutions of the two approximate solutions of  \eqref{e_semi_implementation_8} for {\bf Example 3} as  $\Delta t=1/N\rightarrow 0$  in \eqref{e_semi_implementation_8}. One observes clearly that one solution of  \eqref{e_semi_implementation_8} converges to  the exact solution 1, while the other solution converges to zero. 

\begin{figure}[!htp]
\centering
\includegraphics[scale=0.4]{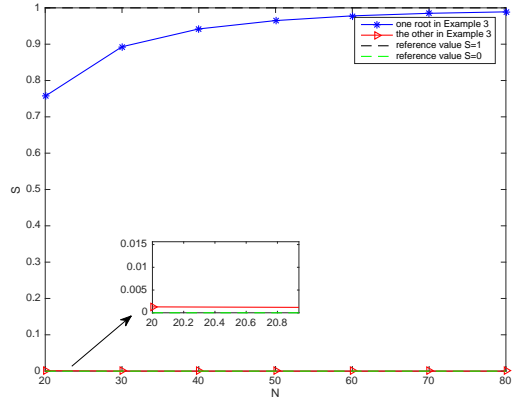}
\caption{Time evolutions of the two approximate solutions of  \eqref{e_semi_implementation_8} as  $\Delta t\rightarrow 0$ for example 3}  \label{fig: S}
\end{figure}

\section*{Appendix\ A\ \ Finite difference discretization on the staggered grids}
\renewcommand\thesection{A}
To fix the idea, we consider $\Omega=(L_{lx},L_{rx})\times (L_{ly},L_{ry})$. Three dimensional rectangular domains can be dealt with similarly.

The two dimensional domain $\Omega$ is partitioned by  $\Omega_x\times \Omega_y$,    where
  \begin{eqnarray}
    & & \Omega_x: L_{lx}=x_{0}<x_{1}<\cdots<x_{N_x-1}<x_{N_x}=L_{rx},\nonumber\\
    & & \Omega_y: L_{ly}=y_{0}<y_{1}<\cdots<y_{N_y-1}<y_{N_y}=L_{ry}.\nonumber
   \end{eqnarray}
For simplicity we also use the following notations:
\begin{equation}    \label{eq:grid-on-bound}
\left\{
\begin{array}{ll}
x_{-1/2}=x_{0}=L_{lx},& x_{N_x+1/2}=x_{N_x}=L_{rx},\\
y_{-1/2}=y_{0}=L_{ly},& y_{N_y+1/2}=y_{N_y}=L_{ry}.
\end{array}
\right.
\end{equation}
For possible integers $i,j$, $0\leq i\leq N_x,\ 0\leq j\leq N_y$, define
    \begin{eqnarray}
& & x_{i+1/2}=\frac{x_{i}+x_{i+1}}{2},\quad  h_{i+1/2}=x_{i+1}-x_{i},\quad   h=\max\limits_{i}h_{i+1/2},\nonumber\\
& & h_{i}=x_{i+1/2}-x_{i-1/2}=\frac{h_{i+1/2}+h_{i-1/2}}{2},\nonumber\\
& & y_{j+1/2}=\frac{y_{j}+y_{j+1}}{2},\quad  k_{j+1/2}=y_{j+1}-y_{j}, \quad   k=\max\limits_{j}k_{j+1/2},\nonumber\\
& & k_{j}=y_{j+1/2}-y_{j-1/2}=\frac{k_{j+1/2}+k_{j-1/2}}{2},\nonumber\\
& & \Omega_{i+1/2,j+1/2}=(x_{i},x_{i+1})\times (y_{j},y_{j+1}).\nonumber
    \end{eqnarray}
It is clear that
$$ \displaystyle
h_{0}=\frac{h_{1/2}}{2},\ h_{N_x}=\frac{h_{N_x-1/2}}{2} ,\ \ k_{0}=\frac{k_{1/2}}{2},\ k_{N_y}=\frac{k_{N_y-1/2}}{2}.$$
For a function $f(x,y)$, let $f_{l,m}$ denote $f(x_l,y_m)$ where $l$ may take values $i,\ i+1/2$ for  integer $i$, and $m$ may take values $j,\ j+1/2$ for  integer $j$. For discrete functions with values at proper nodal-points, define
\begin{equation}
\left\{
\begin{array}{lll}
  \displaystyle [d_{x}f]_{i+1/2,m}=\frac{f_{i+1,m}-f_{i,m}}{h_{i+1/2}},\qquad & \displaystyle  [D_{y}f]_{l,j+1}=\frac{f_{l,j+3/2}-f_{l,j+1/2}}{k_{j+1}}, \\
 \displaystyle [D_{x}f]_{i+1,m}=\frac{f_{i+3/2,m}-f_{i+1/2,m}}{h_{i+1}}, &  \displaystyle  [d_{y}f]_{l,j+1/2}=\frac{f_{l,j+1}-f_{l,j}}{k_{j+1/2}}.
\end{array}
\right. \label{def:difference}
\end{equation}
For functions $f$ and $g$, define some discrete $l^2$ inner products and norms as follows.
\begin{eqnarray}
 (f,g)_{l^2,M} &\equiv &  \sum\limits_{i=0}^{N_x-1}\sum\limits_{j=0}^{N_y-1} h_{i+1/2}k_{j+1/2} f_{i+1/2,j+1/2} g_{i+1/2,j+1/2}, \label{inner:l2-M}\\
      (f,g)_{l^2,T_x} &\equiv &  \sum\limits_{i=0}^{N_x}\sum\limits_{j=1}^{N_y-1} h_{i}k_{j} f_{i,j} g_{i,j}, \label{inner:l2_x}\\
  (f,g)_{l^2,T_y} &\equiv &  \sum\limits_{i=1}^{N_x-1}\sum\limits_{j=0}^{N_y} h_{i}k_{j} f_{i,j} g_{i,j}, \label{inner:l2_y}\\
   \|f\|_{l^2,\xi}^2 &\equiv &  (f,f)_{l^2,\xi},\qquad \xi=M,\ T_x,\ T_y   . \label{norm:l2}
 \end{eqnarray}
Further define discrete $l^2$ inner products and norms as follows.
\begin{eqnarray}
& &(f,g)_{l^2,T,M} \equiv \sum\limits_{i=1}^{N_x-1}\sum\limits_{j=0}^{N_y-1} h_{i}k_{j+1/2} f_{i,j+1/2} g_{i,j+1/2}, \label{inner:l2-T-M}\\
& & (f,g)_{l^2,M,T} \equiv \sum\limits_{i=0}^{N_x-1}\sum\limits_{j=1}^{N_y-1} h_{i+1/2}k_{j} f_{i+1/2,j}g_{i+1/2,j}, \label{inner:l2-M-T}\\
& &  \|f\|_{l^2,T,M}^2 \equiv   (f,f)_{l^2,T,M}, \quad  \|f\|_{l^2,M,T}^2 \equiv  (f,f)_{l^2,M,T}. \label{norm:l2-T-M}
 \end{eqnarray}
For vector-valued functions $\textbf{u}=(u_1,u_2)$, it is clear that
 \begin{eqnarray}
 \|d_x u_{1}\|_{l^2,M}^2 &\equiv &  \sum\limits_{i=0}^{N_x-1}\sum\limits_{j=0}^{N_y-1}h_{i+1/2}k_{j+1/2} |d_x u_{1,i+1/2,j+1/2}|^2, \label{norm:d-x-u-x}\\
 \|D_y u_{1}\|_{l^2,T_y}^2 &\equiv &     \sum\limits_{i=1}^{N_x-1}\sum\limits_{j=0}^{N_y}h_{i}k_{j} |D_y u_{1,i,j}|^2,   \label{norm:d-y-u-x}
 \end{eqnarray}
and $\|d_y u_{2}\|_{l^2,M},\ \|D_x u_{2}\|_{l^2,T_x}$ can be represented similarly.
Finally define the discrete $H^1$-norm and discrete $l^2$-norm of a vectored-valued function $\textbf{u}$,
 \begin{eqnarray}
  \|D \textbf{u}\|^2  & \equiv & \|d_x u_{1}\|_{l^2,M}^2+ \|D_y u_{1}\|_{l^2,T_y}^2+  \|D_x u_{2}\|_{l^2,T_x}^2+\|d_y u_{2}\|_{l^2,M}^2. \label{norm:d-u}\\
  \|\textbf{u}\|_{l^2}^2 &  \equiv & \| u_{1}\|_{l^2,T,M}^2+ \|  u_{2}\|_{l^2,M,T}^2. \label{norm:l2-u}
 \end{eqnarray}
 For simplicity we only consider the case that for all $h_{i+1/2}=h,\ k_{j+1/2}=k$, i.e. uniform meshes are used both in $x$ and $y$-directions. 
 
Finally we present the following useful lemma.
 \begin{lemma}  \label{lemma:U-P-Relation}
\cite{weiser1988convergence} Let $\{V_{1,i,j+1/2}\},\{V_{2,i+1/2,j}\}$ and $\{q_{1,i+1/2,j+1/2}\},\{q_{2,i+1/2,j+1/2}\}$
be discrete functions  with
$V_{1,0,j+1/2}=V_{1,N_x,j+1/2}=V_{2,i+1/2,0}=V_{2,i+1/2,N_y}=0$, with proper integers $i$ and $j$.
 Then there holds
 \begin{equation}
 \left\{
  \begin{array}{lll}
        (D_x q_1,V_1)_{l^2,T,M}&=& -( q_1, d_x V_1)_{l^2,M},\\
        (D_y q_2,V_2)_{l^2,M,T}&=& -( q_2, d_y V_2)_{l^2,M}.
  \end{array}
  \right.
 \end{equation}
\end{lemma}

\bibliographystyle{siamplain}
\bibliography{SAV_MAC}

\end{document}